\documentclass[a4paper,10pt]{article}
\usepackage[english]{babel}
\usepackage{amssymb}
\usepackage{amsmath}
\usepackage{indentfirst}
\usepackage[dvips]{graphicx}
\usepackage{epsfig}
\usepackage{lscape}
\usepackage{hyperref}
\usepackage{color}
\usepackage{soul}
\usepackage[dvipsnames]{xcolor}
\usepackage{amsthm}
\allowdisplaybreaks

\newtheorem{defi}{Definition}[section]
\newtheorem{cor}{Corollary}[section]

\newtheorem{lemma}[cor]{Lemma}
\newtheorem{teo}[cor]{Theorem}
\newtheorem{algo}{Algorithm}[section]
\newtheorem{rem}[cor]{Remark}

\newtheorem{ass}[cor]{Assumption}

\newcommand{\R}{\mathbb{R}}
\newcommand{\N}{\mathbb{N}}

\def\mbN{\mathbb{N}}

\def\bbb{\boldsymbol}

\def\mathref#1{\ifmmode\mathrm{(\ref{#1})}\else{\rm(\ref{#1})}\fi} 
\def\nref#1{\ifmmode\mathrm{\ref{#1}}\else{\rm\ref{#1}}\fi}

\def\eoc{{\rm eoc}}

\def\Bapprox{B_L}
\def\proj{\tilde{P}}
\def\funch{h}

\numberwithin{equation}{section}

\title{A surface finite element scheme \\
for a stochastic PDE on an evolving curve}

\author{Paola Pozzi \thanks{Fakult\"at f\"ur Mathematik, Universit\"at Duisburg-Essen,
Thea-Leymann-Stra\ss e 9,
45127 Essen, Germany, \url{paola.pozzi@uni-due.de}} 
and Bj\"orn Stinner \thanks{Mathematics Institute,
Zeeman Building, 
University of Warwick,
Coventry CV4 7AL, 
United Kingdom, \url{Bjorn.Stinner@warwick.ac.uk}}
}

\begin{document}

\maketitle

\begin{abstract}
 In this paper we consider an ESFEM method for the advection and diffusion of a scalar quantity on a moving closed curve.
 The diffusion process is controlled by a forcing term that may include a rough term (specifically a stochastic noise) which in particular destroys the classical time differentiability properties of the solution. We provide a suitable variational solution concept and a fully discrete FEM discretization. Our  error analysis  appropriately generalizes classical estimates to this weaker setting. We present some numerical simulations that confirm our theoretical findings.
\end{abstract}

\bigskip
 


\section{Introduction}

In this work, we study the numerical approximation of a stochastic partial differential equation (SPDE) of the form
\begin{equation} \label{eq:modeleq}
 dc = \Big{(} -c \frac{|u_x|_t}{|u_x|} + D\frac{1}{|u_x|} \Big{(} \frac{c_x}{|u_x|} \Big{)}_x + w_{T} \frac{c_x}{|u_x|} + r(c) \Big{)} dt + \frac{1}{|u_{x}|} B(c) dW
\end{equation}
using finite elements in space and a semi-implicit Euler-Maruyama time discretization and derive convergence results in expectation subject to suitable regularity assumptions on the solution. Here, $c$ stands for a field on a given evolving closed curve that is parametrised by $u$, $D$ is a diffusion coefficient, $w_{T}$ an advective velocity, $r(c)$ a deterministic reaction and source, and $B(c) dW$ a stochastic reaction and source that involves a (cylindrical) Wiener process $dW$.

Stochastic partial differential equations are used in various application areas including fluids, finance, and population dynamics, see the introduction in \cite{Pardoux} for some examples and \cite{SigKueSta2015,KraGomBakEtAl2018} for modelling aspects. Our study is specifically motivated by the model for cell motility in \cite{EllStiVen2012}. There, noise was added to the problem after discretization to accelerate a cell's decisions about its motion direction in order to fit the simulations with experimental data. However, a continuous model framework was missing. We here address this questions with respect to the biochemistry, which is described by systems of reaction diffusion equations in \cite{EllStiVen2012} that, with suitable noise terms, would be of the form \eqref{eq:modeleq}. 

Computational methods for SPDEs such as \eqref{eq:modeleq} often are based on method-of-lines type approaches.
The resulting systems of stochastic ordinary differential equations then can be solved with standard approaches, where Euler-Maruyama and Milstein are very popular (see \cite{Pla1999} for an introduction and overview). The integration of space and time discretization is discussed in \cite{Powell}, noting that the theory there is developed for semi-linear equation that allow for mild solutions. For more general results on the numerical approximation of SPDEs, specifically nonlinear SPDEs, we refer to \cite{OndProWal2023}. Here, we consider the approximation of the evolving curve by a polygon obtained by interpolating $u$ with standard piece-wise linear continuous finite element functions, so that also the operators are approximated only. This procedure is motivated by the fact that the coupled problem in \cite{EllStiVen2012} involves finding the evolving curve, too, for which finite elements are used.

In their seminal paper \cite{DE07}, the authors provide a finite element method to approximate PDEs on given, evolving hypersurfaces, where the impact of the approximation of the surface is discussed. In particular their weak solution is weakly differentiable in time. However, rough sources or reactions typically destroy the regularity properties, specifically with respect to time to the extent that a time derivative does not exist any more. We follow ideas and concepts in \cite{RoecknerBuch}, where a time-integrated version of the SPDE is used so that no time derivatives appear any more. Moreover, a variational setting is employed that is amenable to finite elements. Well-posedness is guaranteed and includes a stability estimate \emph{on average} (in expectation), and thus provides a framework within which a numerical approximation can be endeavoured. Our objective is to provide a finite element (FE) analysis that relies on a minimum of regularity in the spatial variable as provided by the natural stability estimate, namely the existence of a first weak derivative only.

With regards to the time regularity we assume as in \cite{BDE} that the (stochastic) forcing terms still allow for a solution for which some H\"{o}lder-type estimates are fulfilled in expectation. However, our aim is also to ensure that, if the solution happens to be sufficiently regular so that a time derivative exists, then classical estimates such as in \cite{DE07,DE12,DEL2} can be retrieved. For instance, our spatial analysis uses standard interpolation estimates and a Ritz projection so that an improvement of the convergence rate in the spatial step size is in reach. Nevertheless, as a byproduct of our time-integrated approach, no concept of a time derivative (such as a material derivative) is required in the deterministic case.

Detail on the SPDE \eqref{eq:modeleq}, the variational framework, the a-priori estimate, and the regularity assumptions are presented and discussed in Section 2. The finite element approximation and the time stepping scheme are introduced in Section 3. Our novel scheme is fully discrete and linear and consistent with previous approaches in that it would result in schemes from the literature if the noise term was trivial ($B=0$).
We then discuss some results such as a suitable Ritz projection and state and prove the main convergence results. With respect to the $L^2$ norm in space and in expectation we provide upper estimates for convergence including rates in terms of the spatial step size, the time step size, the approximation of the stochastic noise term, and the approximation of the initial data. Our main result is summarised in Theorem \ref{err-est-g}.
To balance some terms arising from the (lack of) time regularity we have to choose the spatial step size proportional to the time step size, though. This assumption can be dropped if the solution happens to have better regularity properties, which we discuss in detail at the end of Section 3.
We finish the paper with some computational results in Section 4.

\section{Variational Stochastic SPDE on an Evolving Curve}

\subsection{A PDE on a moving curve}
In the following, we use the mathematical language typical for treatment of problems formulated on hypersurfaces in $\R^{n}$ as the questions we explore can be asked also in this more general setting. However, we start our investigation in the framework of moving curves in order to avoid the challenges posed by the higher dimensional setting.

Consider a family of planar embedded 
closed curves $\{\Gamma(t)\}_{t \in [0,T]}$.
The (time dependent) curvature, unit tangent and normal vectors are denoted by $\kappa(t)$, $\tau(t)$, $\nu(t)$ respectively. We suppose that the curves move with a velocity
$$ v = v_{T} + v_{\nu} $$
where $v_{\nu}$ denotes the normal component of the velocity vector $v$ (and thus of the evolving curve $\Gamma$) and $v_{T}$ stands for the tangential component.

For a scalar field $c(t)$ on $\Gamma(t) \subset \R^{2}$ for all $t$ we consider the following partial differential equation:
\begin{equation}\label{eq0}
 \partial_{t}^{\bullet (v)} c + c \nabla_{\Gamma} \cdot v - \triangle_{\Gamma} c = w \cdot \nabla_{\Gamma} c + r(c) \qquad \text{ on } \Gamma(t) \, \,\forall t.
\end{equation}
Here, $\partial_{t}^{\bullet (v)} c = \partial_{t} c + v \cdot \nabla c$ is the material (time) derivative with respect to the velocity $v$, $\nabla_{\Gamma}$ is the (spatial) surface gradient on $\Gamma(t)$ (the $t$-dependence is dropped for a shorter notation), $\triangle_{\Gamma}$ is the Laplace-Beltrami operator, $w$ is a sufficiently smooth given advection field.
Finally, $r : \R \to \R$ is a reaction term in which we will incorporate a source of noise leading to a lack of regularity later on in Section \ref{sec:stochset}.
In addition, we impose the initial condition
\begin{equation} \label{eq0ic}
 c(0,y)=c_{0}(y) \quad \forall y \in \Gamma(0)
\end{equation}
with sufficiently smooth given initial data $c_{0} : \Gamma(0) \to \R$.

The equation \eqref{eq0} can be understood as a balance equation for a material quantity that is transported with the curve, by diffusion and with an additional material velocity along the curve, and is subject to reactions.

For smooth PDEs on moving surfaces without random terms it has proved convenient to consider suitable Sobolev spaces on the evolving surface (see for instance \cite{ranner} and references given in there; see also \cite{DEKR} for PDEs with {\it mild random} terms). In our context, however, it is preferable to transform \eqref{eq0} to a spatial domain that is independent of time. For this purpose, we assume that $\Gamma(t) = u(t,S^{1})$ is parametrized by a smooth given map $u: I \times S^{1} \to \R^{2}$, where $I = [0,T]$ for some $T \in (0,\infty)$. The family of parameterizations is assumed regular in the sense that
\begin{align}\label{regularity}
 0<d_{0} \leq |u_{x}(t,x)| \leq \frac{1}{d_{0}} \qquad \text{ uniformly on } [0,T] \times S^{1}
\end{align}
for some $d_0\in\R^{+}$. Moreover, we assume that $u(t)$ is an embedding at every time and that
\[
 v(t,u(t,x)) = \partial_t u(t,x) \quad \forall (x,t) \in [0,T] \times S^{1}.
\]
The surface PDE \eqref{eq0} then becomes
\begin{equation}\label{eq1}
c_t + c \frac{|u_x|_t}{|u_x|} - \frac{1}{|u_x|} \Big{(} \frac{c_x}{|u_x|} \Big{)}_x = w_{T} \frac{c_x}{|u_x|} + r(c)
\end{equation}
where $w_{T} = w \cdot \frac{u_x}{|u_x|}$ and where we, with a slight abuse of notation, used $c : [0,T] \times S^{1} \to \R$ again to denote the same field but on the fixed spatial domain now, thus writing $c(t,x)$ for $c(t,u(t,x))$.
For a variational formulation we introduce the following Gelfand triple: let
\begin{align}\label{VH}
V:= W^{1,2}(S^{1}, \R), \qquad H := L^{2}(S^{1}, \R)
\end{align}
denote the usual Sobolev spaces that are periodic on $[0,2 \pi]\simeq S^{1}$. We note that $V \subset H \simeq H' \subset V'$, with $V$ densely and compactly embedded in $H$. 
The inner products on $H$ and $V$ are denoted by $\langle \cdot, \cdot \rangle_{H}$ and $\langle \cdot, \cdot \rangle_{V}$ respectively.

Multiplying equation \eqref{eq1} with $|u_x|$ and a test function $\varphi \in V$, and then integrating with respect to the spatial variable we obtain that
\begin{equation}\label{eq2}
\frac{d}{dt} \Big{(} \langle c |u_{x}|, \varphi \rangle_{H} \Big{)}
+ \langle \frac{c_{x}}{|u_{x}|}, \varphi_{x} \rangle_{H}
= - \langle c w_{T}, \varphi_{x} \rangle_{H} - \langle c \partial_x w_{T}, \varphi\rangle_{H} + \langle r(c) |u_{x}|, \varphi \rangle_{H}.
\end{equation}
Note that only first order spatial derivatives feature, thus enabling the use of standard continuous, piece-wise linear finite elements. These will be used not only for $c$ but also for the parametrization $u$, which corresponds to the approximation of the evolving curve by a polygon. In the following, we absorb the term $c \partial_x w_{T}$ into the reaction term $r(c)$.

The noise in the reaction term will lead to a loss in regularity such that $c$ has no time derivative any more. A suitable weak formulation is obtained by integrating with respect to time from $0$ to an end time denoted with $t$ again:
\begin{multline}\label{eq3}
\langle c(t) |u_{x}(t)|, \varphi \rangle_{H} - \langle c_{0} |u_{x}(0)|, \varphi \rangle_{H} + \int_{0}^{t} \langle \frac{c_{x}(t')}{|u_{x}(t')|}, \varphi_{x} \rangle_{H} dt'
\\
= \int_{0}^{t} - \langle c(t') w_{T}(t'), \varphi_{x} \rangle_{H} + \langle r(c(t')) |u_{x}(t')|, \varphi \rangle_{H} dt'
\end{multline}
for all $t \in [0,T]$ and test functions $\varphi \in V$.

\begin{rem} We remark the following on the above weak formulation:
\begin{itemize}
 \item
 If the test functions $\varphi$ are differentiable functions of time then the weak formulation of the ESFEM (Evolving Surface Finite Element Method) from \cite[Definition 4.1]{DE07} can be obtained, which in our case with the parametrization reads
 \begin{equation}\label{eq2_time}
 \frac{d}{dt} \Big{(} \langle c |u_{x}|, \varphi \rangle_{H} \Big{)} -\langle c |u_x|, \varphi_t \rangle_{H}
 + \langle \frac{c_{x}}{|u_{x}|}, \varphi_{x} \rangle_{H}
 = - \langle c w_{T}, \varphi_{x} \rangle_{H} + \langle r(c) |u_{x}|, \varphi \rangle_{H}.
 \end{equation}

 \item
 In \cite{DEKR} the ESFEM is extended to random evolving surface finite element methods for the advection-diffusion equation
 $$ \partial^{\bullet(v)}_{t} c + c \nabla_{\Gamma} \cdot v -\nabla_{\Gamma} \cdot (\alpha \nabla_{\Gamma} c) = f $$
 on an evolving compact hypersurface $\Gamma(t)$ in $\R^{n}$, where $\alpha$ is a uniformly bounded random coefficient and $v$ still deterministic. The random coefficient $\alpha$ is such that (see \cite[Assumption 2.1]{DEKR}) the concept of path-wise material derivative $\partial^{\bullet(v)}_{t}$ is still applicable to the solution in the weak formulation. \\
 In our setting, a solution (more precisely a sample path $t \to c(\omega, \cdot) $) usually does not admit any weak derivative in time, hence the (necessity of an) integrated formulation.
\end{itemize}
\end{rem}

\subsection{Reactions with noise}
\label{sec:stochset}

Our approach to the noisy reaction term is based on Wiener processes in Hilbert spaces. We follow the variational approach to SPDEs along the lines of \cite{RoecknerBuch} but, for further background, detail, and proofs, also refer to \cite{Powell,DaPratoBuch}.

To account for stochastic effects we neglect any deterministic component (since for these the error analysis is well known) and rewrite the reaction term (last term $\int_{0}^{t} \langle r(c(t')) |u_{x}(t')|, \varphi \rangle_{H} dt'$ in \eqref{eq3}) as a stochastic integral in the form
\begin{equation} \label{def_stochreact}
 \int_0^t \langle B(c(t')) dW(t'), \varphi \rangle_{H}.
\end{equation}
For the definition of such integrals it is convenient to introduce another separable Hilbert space $U$ with a orthonormal basis $(g_l)_{l \in \mbN}$, and to consider bounded, linear Hilbert-Schmidt operators $\Phi \in L_{2}(U,H)$, which satisfy $\| \Phi \|_{L_{2}(U,H)}^2 = \sum_{l \in \mbN} \| \Phi g_l \|_H^2 < \infty$.

We assume that $W(t)$, $t \in [0,T]$ is a cylindrical $Q$-Wiener process with $Q:=I$ taking values in $U$ and being defined on a complete probability space $(\Omega,\mathcal{F}, \mathbb{P})$ with normal filtration $\mathcal{F}_{t}$, $t \in [0,T]$. In particular, $W$ has the representation
\begin{equation} \label{repW}
 W(t) = \sum_{l \in \N} g_{l} \beta_{l}(t)
\end{equation}
where the $(\beta_{l})_{l \in \N}$ are mutually independent real-valued $\mathcal{F}_{t}-$adapted Brownian motions.

We furthermore assume that
\begin{equation} \label{OPB}
B : V \to L_{2}(U,H)
\end{equation}
is a continuous map with the following properties:
\begin{itemize}
 \item
 there exist constants $ C_{B} \geq 0$
 such that for all $c \in V$ 
 \begin{align} \label{BH3}
 \| B(c) \|^{2}_{L_{2}(U,H)} \leq C_{B}(1 + \| c \|^{2}_{H}).
 \end{align}
 \item
 $B$ is weak monotone: there exists $C_{BM}\geq 0$ such that for all $c_{1}, c_{2} \in V$
 \begin{align}\label{BH2}
 \| B(c_{1}) - B(c_{2}) \|^{2}_{L_{2}(U,H)} \leq C_{BM}\| c_{1} - c_{2} \|^{2}_{H}.
 \end{align}
\end{itemize}
Note that by \eqref{BH3}, $B$ satisfies also the following growth bound:
\begin{align}\label{BH4}
 \| B( c) \|^{2}_{L_{2}(U,H)} \leq C_{B} (1 + \| c \|^{2}_{H}) \leq C_{B} (1 + \| c \|^{2}_{V})
\end{align}
for all $c \in V$.
Consequently, \eqref{def_stochreact} is well-defined and can be written as
\begin{equation} \label{def_stochreact2}
 \int_0^t \langle B(c(t')) dW(t'), \varphi \rangle_{H} = \sum_{l \in \mbN} \int_0^t \langle B(c(t')) g_l, \varphi \rangle_{H} d \beta_l(t').
\end{equation}

Motivated by the variational formulation \eqref{eq3} of a PDE on a moving curve and accounting for the definition \eqref{def_stochreact} of the noisy reaction term we define solutions as follows:

\begin{defi} \label{def_solution}
Let
$V$, $H$ be as in \eqref{VH}. Let $W(t)$, $t \in [0,T]$, be a cylindrical Wiener process as defined around \eqref{repW}.
Moreover, let $c_{0} \in L^{2}(\Omega; H)$ be $\mathcal{F}_{0}$-measurable. A continuous $H$-valued $\mathcal{F}_t$-adapted process $(c(t))_{t\in [0,T]}$ with $c \in L^{2}(\Omega;L^{2}((0,T);V)) \cap L^{2}(\Omega;C^{0}([0,T],H))$ is called a solution of \eqref{eq0} with stochastic reaction term \eqref{def_stochreact2} and of \eqref{eq0ic} if
\begin{multline} \label{eq:def_solution}
\langle c(t) |u_x|, \varphi \rangle_{H} - \langle c_{0} |u_x(0)|, \varphi \rangle_{H} + \int_{0}^{t} \langle \frac{c_{x}(t')}{|u_{x}(t')|}, \varphi_{x} \rangle_{H} dt'
\\
= - \int_{0}^{t} \langle c(t') w_{T}(t'), \varphi_{x} \rangle_{H} dt' + \sum_{l=1}^{\infty} \int_{0}^{t} \langle B (c(t')) g_{l} , \varphi \rangle_{H} d \beta_{l}(t')
\end{multline}
$\mathbb{P}$-a.s. for all $ t \in [0,T]$ and for all $\varphi \in V$.
\end{defi}

\subsection{Solutions and a priori estimates}
\label{sec:existenceSol}

We briefly discuss existence of solutions and a priori estimates. In order to deal with the presence of the length element $|u_{x}(t)|$
it is sensible to include it in the operator itself. To this end we introduce the variable
\begin{align}\label{eq-chc} \hat{c}(\omega,t,x):=c(\omega,t,x)|u_{x}(t,x)|, \qquad \quad \omega \in \Omega, \, t \in [0,T], \, x \in S^{1},
\end{align}
and rewrite \eqref{eq:def_solution} as
\begin{align} \label{eq:hatc}
\langle \hat{c}(t), \varphi \rangle_{H} - \langle \hat{c}_0, \varphi \rangle_{H}
& =-\int_{0}^{t} \langle \hat{A}(t', \hat{c}(t')), \varphi \rangle_{V',V} dt' + \sum_{l=1}^{\infty}\int_{0}^{t} \langle B (\frac{\hat{c}(t')}{|u_{x}(t')|}) g_{l} , \varphi \rangle_{H} d \beta_{l}(t')
\end{align}
where the linear operator $\hat{A}(t, \cdot) : V \to V'$, $t \in [0,T]$, is given by
\begin{align}\label{eqAtilda}
 \langle \hat{A}(t,\eta), \varphi \rangle_{V',V} :=
 \Big{\langle} \frac{\eta_{x}}{|u_{x}(t)|}, \frac{\varphi_{x}}{|u_{x}(t)|} \Big{\rangle}_{H}
 - \Big{\langle} \eta \big{(} \frac{u_{x}(t) \cdot u_{xx}(t)}{|u_{x}(t)|^3} - w_{T} \big{)}, \frac{\varphi_{x}}{|u_{x}(t)|} \Big{\rangle}_{H}, \quad \eta,\varphi \in V.
\end{align}
Observe that thanks to the regularity assumptions on $w$ and on the parametrisation around \eqref{regularity} the deterministic operator $\hat{A}(t,\cdot)$ is bounded,
\begin{align} \label{hatAbounded}
\| \hat{A}(t,\eta) \|_{V^{'}} \leq \hat{C} \| \eta \|_{V} \text{ for all } t \in [0,T], \eta \in V,
\end{align}
for some constant $\hat{C} > 0$ depending on the data only (including $u$). Moreover, it is coercive in the sense that there are constants $\hat{\alpha} > 0$ and $\hat{\beta} > 0$ such that
\begin{align} \label{hatAcoercive}
\langle \hat{A}(t,\eta), \eta \rangle_{V',V} \geq \hat{\alpha} \| \eta \|^{2}_{V} - \hat{\beta} \| \eta \|^{2}_{H}, \text{ for all } t \in [0,T], \eta \in V.
\end{align}
Note that from \eqref{hatAcoercive} we infer for any $t \in [0,T]$ and $\eta_{1},\eta_{2} \in V$ that
\begin{align}\label{weakmonAhat}
- \langle \hat{A}(t,\eta_{1}-\eta_{2}), (\eta_{1}-\eta_{2}) \rangle_{V',V} \leq \hat{\beta} \| \eta_{1}-\eta_{2}\|_{H}^{2}.
\end{align}
The assumptions of Theorem 4.2.4 in \cite{RoecknerBuch} are satisfied:
Noting that $\hat{A}$ has the opposite sign to $A$ in \cite{RoecknerBuch} and recalling the regularity of $u$ and, specifically, the length element $|u_{x}(t)|$ around \eqref{regularity}, (H1) is satisfied because $\hat{A}$ is linear in $\eta$, (H2) is satisfied thanks to \eqref{weakmonAhat} and \eqref{BH2}, (H3) is satisfied thanks to \eqref{hatAcoercive} and \eqref{BH3}, and (H4) is satisfied with $\alpha = 2$ thanks to \eqref{hatAbounded}. We conclude:

\begin{teo}\label{thm-exc}
There exists a unique solution in the sense of Definition~\ref{def_solution}. 
\end{teo}
Uniqueness means that if two solutions $c_{1},c_{2}$ are found then
$\mathbb{E}(\|c_{1}(t) -c_{2}(t))\|^{2}_{H})=0$ holds for all $t \in [0,T]$ (see \cite[Proposition~4.2.10]{RoecknerBuch}).

\begin{rem}
 The smoothness of the length element $|u_{x}(t)|$ (see \eqref{eqAtilda}) is essential to ensure that the assumptions are satisfied. If the curve is approximated by something of lower regularity, such as a polygon in a finite element context, then we can no longer guarantee these properties and the idea in \eqref{eq-chc} around absorbing the length element in the solution no longer works. For the finite element analysis we therefore work directly with the formulation of Definition~\ref{def_solution}.
\end{rem}


A direct consequence of the well-posedness theorem is the following a-priori estimate for the expected value of the solution:
\begin{cor}\label{c-apriori-est}
Let $c$ be the solution according to Definition~\ref{def_solution}. Then
\begin{align*}
\sup_{ t \in [0,T]} \mathbb{E} (\| c(t)\|^{2}_{H}) + \int_{0}^{T} \mathbb{E}(\| c(t') \|^{2}_{V}) dt' \leq C.
\end{align*}
for some positive constant $C=C(\|c_{0}\|_{H},T,u).$
\end{cor}

\begin{proof}
Applying the It\^{o} formula 
 from \cite[Theorem 4.2.5]{RoecknerBuch} to \eqref{eq:hatc} yields that
\begin{multline*}
 \| \hat{c}(t) \|_{H}^2 + 2 \int_0^t \langle \hat{A}(t',\hat{c}(t')),\hat{c}(t') \rangle_{V',V} dt' \\
 = \| c_0 \|_{H}^2 + \int_0^t \Big{\|} B \big{(} \frac{\hat{c}(t')}{|u_{x}(t')|} \big{)} \Big{\|}_{L_2(U,H)}^2 dt' + \sum_{l=1}^{\infty} \int_0^t 2 \langle B ( \frac{\hat{c}(t')}{|u_{x}(t')|} ) g_l, \hat{c}(t') \rangle_{H} d \beta_{l}(t').
\end{multline*}
We now use \eqref{hatAcoercive}, \eqref{BH3} and the fact that ``Brownian motions average zero'' (see \cite[Remark~4.2.8]{RoecknerBuch} for details) to conclude that
\[
 \mathbb{E} \big{(} \| \hat{c}(t) \|_{H}^2 \big{)} + \int_0^t 2 \hat{\alpha} \mathbb{E} \big{(} \| \hat{c}(t') \|_{V}^2 \big{)} dt' \leq \mathbb{E} \big{(} \| c_0 \|_{H}^2 \big{)} + C \int_0^t 1 + \mathbb{E} \big{(} \| \hat{c}(t') \|_{H}^2 \big{)} dt'.
\]
The corollary follows from a Gronwall argument and using the regularity of $u$ around \eqref{regularity}.
\end{proof}

\subsection{Regularity assumptions}

To perform an error analysis for a discrete approximation, stronger (smoothness) assumptions on the noise term and on the solution are required.

Regularity results (which are outside the scope of this paper) typically require some conditions on the stochastic and deterministic forcing terms, the regularity of the initial data, the regularity of the coefficients, and some compatibility conditions. Regularity theory is, as well known in the (S)PDE community, a very delicate matter. With regards to the question of space regularity we refer for example to \cite{Zhang07}, \cite{Hof13} and references given in there. Concerning time regularity we mention (in the context of mild solutions) \cite{KruseLarsson12}, \cite{Kruse14} and refer also to the references given in there. For our error analysis of a finite element approximation we make the following assumption:
\begin{ass}\label{lass-cholder}
There exists $\nu_{r} \in [0,1)$ such that
 \begin{align}\label{ass-cholder}
 \sup_{t, \tau \in [0,T], t \neq \tau} \frac{(\mathbb{E}(\| c(t) - c(\tau) \|_{V}^{2}))^{\frac{1}{2}}}{|t -\tau|^{\frac{\nu_{r}}{2} }} < \infty
 \end{align}
\end{ass}

\begin{rem}\label{remNU1}
A couple of remarks on the above regularity assumption on the solution:
\begin{itemize}
 \item
 Our assumption is strongly motivated by \cite[Theorem~2.2]{BDE}, where it is shown that in a similar variational setting \eqref{ass-cholder} is satisfied by the stochastic heat equation if the initial data are sufficiently smooth and the stochastic operator $B$ has some additional properties.

 \item
 If $u$ is such that $|u_{x}|=1$ (for instance, a parametrisation of the stationary unit circle) then $c$ is the solution to the stochastic heat equation with periodic boundary conditions.
 For that equation, several regularity results are known, for instance, \cite[\S~10.4]{Powell} in the context of mild solutions with Dirichlet boundary conditions. Such results usually are subject to additional regularity assumptions on the forcing term (see \cite[Assumption~10.23]{Powell}).

 \item
 If $c$ admits a classical $L^2$ time derivative with values in some Sobolev space $W^{s,2}$ then we can write (without loss of generality let $t > \tau$)
 \begin{align*}
 \| c(t) - c(\tau) \|_{W^{s,2}(S^{1})} &=\|\int_{\tau}^{t} c_{t}(t') dt' \|_{W^{s,2}(S^{1})}\leq
 \int_{\tau}^{t}\| c_{t}(t') \|_{W^{s,2}(S^{1})} dt'\\
 &\leq \sqrt{|t -\tau|} (\int_{\tau}^{t}\| c_{t}(t') \|_{W^{s,2}(S^{1})}^{2} dt')^{\frac{1}{2}}
 \end{align*}
 and therefore \eqref{ass-cholder} is quickly verified.
 
\end{itemize}
 \end{rem}

We also need to make some assumptions on the approximation of the stochastic perturbation.
Since $B$ acts on a process $W(t)$ in an infinite dimensional space, it is natural to consider projections to finite dimensional subspaces for computations. Thus, denote by $\proj_{L}: U \to {\rm span} \langle g_{1}, \ldots, g_{L} \rangle_{U}$ the orthogonal projection given by
\begin{align}\label{P-tilde}
 \proj_{L} u = \sum_{l=1}^{L} \langle u, g_{l} \rangle_{U} g_{l} \qquad \qquad u \in U.
\end{align}

The idea is that if $L$ is large then the error we make in considering $B(c) \proj_{L}$ instead of $B(c)$ is small. This is indeed true pointwise by definition of any projection operator. For the analysis that follows however we assume that this is uniformly the case:

\begin{ass}\label{approximW}
Given any $\epsilon_{W}>0$ we can choose $L$ so large that
\begin{align*} 
 \mathbb{E}(\|B(y) (Id - \proj_{L})\|^{2}_{L_{2}(U,H)}) \leq \epsilon_{W} (1+ \mathbb{E}(\| y \|^{2}_{H})) \quad \forall y \in L^{2}(\Omega,V).
\end{align*}
\end{ass}

Analogous assumptions can be found in \cite[Assumption~10.32]{Powell} and \cite[(H4) in \S4.3]{BDE}, which are motivated by \cite[Lemma 10.33]{Powell} and \cite[remark 4.3]{BDE} respectively.

\section{Numerical Approximation and Convergence Analysis}

\subsection{Discretization}

We now discretize the equation both in space and time, using piecewise linear finite element for the space variable and a Euler-Maruyama scheme for the time variable.

For the spatial discretization let $S^1= \bigcup_{j=1}^N S_j$ be a decomposition of $S^1$ into segments given by nodes $x_j$. We think of $S_j$ as the interval $[x_{j-1}, x_{j}] \subset [0,2\pi]$ for $j=1, \ldots, N$. In the following, indices related to the grid have to be considered modulo $N$. For instance, we identify $x_0=x_N$. Let $h_j=|S_j|$ and $h=\max_{j=1,\ldots, N}h_j$ be the maximal diameter of a grid element. We assume that for some constant $\bar{C}>0$ we have
\begin{align}
\label{(4.1)}
h_j \geq \bar{C}h \quad \forall j.
\end{align}
Clearly, the above inequality yields that $\bar{C}h_{j+1} \leq h_j \leq \frac{h_{j+1}}{\bar{C}}$ for all $j$.
For a discretization of the (given) geometric evolution we introduce the discrete finite dimensional spaces
\begin{equation*}
S_h := \{ v \in C^0(S^1, \R) \, : \, v|_{S_j} \in P_1 (S_j),\, j=1 \cdots, N\} \subset V \subset H, \qquad X_h = S_h^2,
\end{equation*}
of continuous periodic piecewise affine functions on the grid. The scalar nodal basis functions of $S_h$ are denoted by $\phi_j$, $j=1, \ldots, N$, and characterized by $\phi_j(x_i)=\delta_{ij}$.

For a continuous function $v \in C^0 (S^1, \R)$ let $I_h v \in S_h$ be its linear interpolation that is uniquely defined by $I_h v(x_i)= v(x_i)$ for all $i=1, \ldots,N$. For convenience we also denote the interpolation onto $X_h$ by $I_h$. We note the following standard interpolation estimates (both for scalar and vector valued functions):
\begin{align}
\label{(4.2)}
\| v-I_h v \|_{L^2(S^1)} &\leq C h^k \|v \|_{W^{k,2}(S^1)} &\text{ for $k=1,2$}\,,\\
\label{(4.2)bis}
\| (v-I_h v)_x \|_{L^2 (S^1)} & \leq C h \| v \|_{W^{2,2} (S^1)}\,,\\
\label{(4.2)extra}
 \| (I_{h} v)_{x} \|_{L^{2} (S^{1})} & \leq C \| v_{x} \|_{ L^{2}(S^{1})}\, . 
\end{align} 
Recall also the following inverse estimates for any $v_{h} \in S_{h}$ and $j=1, \ldots, N$:
\begin{align}
\| v_{hx} \|_{L^{2}(S_{j})} & \leq \frac{C}{h_{j}} \| v_{h} \|_{L^{2} (S_{j})} & \quad \overset{\eqref{(4.1)}}{\Longrightarrow} \quad \|v_{hx} \|_{L^{2}(S^1)} & \leq \frac{C}{h} \| v_{h} \|_{L^{2} (S^1)}, \label{IE-1} \\
\| v_{h} \|_{L^{\infty} (S_{j})} & \leq \frac{C}{\sqrt{h_{j}}} \| v_{h} \|_{L^{2}(S_{j}) } & \quad \overset{\eqref{(4.1)}}{\Longrightarrow} \quad \| v_{h} \|_{L^{\infty} (S^1)} & \leq \frac{C}{\sqrt{h}} \| v_{h} \|_{L^{2}(S^1)} \label{IE-2}.
\end{align}
Despite the fact that all norms are equivalent in $S_{h}$, to indicate the scalar product that we are considering on the finite dimensional spaces we introduce the notation
\begin{align*}
V_{h}:= (S_{h}, \langle, \rangle_{V}), \qquad \qquad H_{h}:= (S_{h}, \langle, \rangle_{H}).
\end{align*}

Next, let $0=t_{0} < t_{1} < \ldots < t_{M}=T$ be an equi-partition of the time interval $[0,T]$ and set $\triangle t= t_{k}-t_{k-1}$ for $k=1, \ldots, M$.
For the given map $u=u(t,x)$ we write
$$ u^{k}(x)= u(t_{k},x), \qquad x \in S^{1},$$
and set
$$ u_{h}^{k}(x):= I_{h} u(t_{k},x), \qquad w^{k}_{Th}(x):= I_{h} w_{T}(t_{k},x) \qquad x \in S^{1}. $$
Note that by the regularity of $u$, \eqref{regularity}, and using \eqref{IE-2} we can infer that for $h \leq h_{0}$ sufficiently small (for instance, see \cite{PozziStinner2017} for the details)
\begin{align}\label{regularityh}
\frac{d_{0}}{2} \leq |u_{hx}^{k}(x)| \leq \frac{2}{d_{0}} \quad \text{ for } x \in S^{1} \text{ and for any } k.
\end{align} 
Moreover, by the smoothness properties of $w_{T}$ we have that
\begin{align}\label{bound-wth}
\| w^{k}_{Th} \|_{L^{\infty}(S^{1})} \leq C\| w_{T}(t_{k}, \cdot)\|_{L^{\infty}(S^{1})} \leq C \qquad \text{ for any }k.
\end{align}

We seek for an approximation of $c(\omega, t_{k},x)$ of the form
$$c^{k}_{h}(\omega,x)= \sum_{j=1}^{N} c_{j}^{k}(\omega) \phi_{j}(x), \qquad c^{k}_{h}(\omega,\cdot )\in V_{h} \quad \forall \omega \in \Omega$$
with coefficients maps $c_{j}^{k}: \Omega \to \R$ and propose the following fully discrete scheme:

\begin{algo}\label{fullydiscrete-scheme-g}
Given an $\mathcal{F}(t_{0})$-measurable initial map $c^{0}_{h} \in L^{2}(\Omega, S_{h})$, iteratively find maps $c^{k}_{h} \in L^{2}(\Omega, S_{h})$, $k=1,2, \ldots, M$ such that
\begin{align}\label{eq-3.6-g}
\langle c^{k}_{h} |u^{k}_{hx}|, \varphi_{h} \rangle_{H} + \triangle t \langle \frac{c^{k}_{hx}}{| u^{k}_{hx}|}, \varphi_{hx} \rangle_{H}
&= \langle c^{k-1}_{h} | u^{k-1}_{hx}|, \varphi_{h} \rangle_{H}
-\triangle t \langle c^{k-1}_{h} w_{Th}^{k}, \varphi_{hx} \rangle_{H} \notag\\
& \quad + \langle \Bapprox(c^{k-1}_{h})\triangle W_{k}, \varphi_{h}\rangle_{H}
\qquad \forall \varphi_{h} \in V_{h},
\end{align}
where (upon recalling~\eqref{P-tilde})
$\triangle W_{k}= W(t_{k})-W(t_{k-1}) \qquad \text{ and } \qquad \Bapprox(c):= B(c)\proj_{L}.$
\end{algo}

Note that
\begin{align*}
\langle \Bapprox(c^{k-1}_{h})\triangle W_{k}, \varphi_{h}\rangle_{H}
&= \sum_{l=1}^{L} \langle B(c^{k-1}_{h})(g_{l}), \varphi_{h} \rangle_{H} \triangle \beta_{l,k}
=\sum_{l=1}^{L} \langle B(c^{k-1}_{h})(g_{l}), \varphi_{h} \rangle_{H} (\beta_{l}(t_{k})-\beta_{l}(t_{k-1}))\\
&= \langle \int_{t_{k-1}}^{t_{k}} \Bapprox(c^{k-1}_{h}) dW(t'), \varphi_{h} \rangle_{H}
= \langle \int_{t_{k-1}}^{t_{k}} B(c^{k-1}_{h}) \proj_{L} \,dW(t'), \varphi_{h} \rangle_{H}.
\end{align*}

Writing $c_h^k(\omega,x) = \sum_{j=1}^{N} c^{k}_{j}(\omega) \phi_j(x)$, $(\omega,x) \in \Omega \times S^1$, the above equation \eqref{eq-3.6-g} is equivalent to a linear system
\begin{align}
(\bbb{M}^{k} + \triangle t\, \bbb{S}^{k})\bbb{c}^{k}(\omega)= \bbb{f}^{k-1}(\omega). \label{eq:linsys}
\end{align}
Here, $\bbb{c}^{k}(\omega) = (c^{k}_{1}(\omega), \ldots, c^{k}_{N}(\omega)) \in \R^N$, the symmetric matrices $\bbb{M}^{k}, \bbb{S}^{k} \in \R^{N \times N}$ have entries
$$ M_{ij}^{k} = \int_{S^{1}} \phi_{i}(x) \phi_{j}(x) |\partial_{x}(I_{h} u (t_{k}, x))| dx, \qquad
S_{ij}^{k} = \int_{S^{1}} \partial_{x}\phi_{i}(x) \partial_{x}\phi_{j}(x) \frac{1}{| \partial_{x}(I_{h} u (t_{k}, x))|} dx, $$
and the right-hand side $\bbb{f}^{k-1}(\omega) = (f_{1}^{k-1}(\omega), \ldots, f_{N}^{k-1}(\omega)) \in \R^N$ has entries
$$ f_{i}^{k-1}(\omega) = \langle c^{k-1}_{h} | u^{k-1}_{hx}|, \phi_{i} \rangle_{H}
-\triangle t \langle c^{k-1}_{h} w_{Th}^{k}, \partial_{x} \phi_{i} \rangle_{H}
+ \langle \Bapprox(c^{k-1}_{h}) \triangle W_{k}, \phi_{i} \rangle_{H}.$$

We note that the matrix $\bbb{M}^{k} + \triangle t \bbb{S}^{k}$ is positive definite so that
\eqref{eq:linsys} has a unique solution for any given right-hand side $\bbb{f}^{k-1}(\omega)$.
Moreover, recalling the assumptions on the Wiener process around \eqref{repW} 
it follows that the coefficients
$c^{k}_{j}: \Omega \to \R$, $j=1, \ldots, N$, and therefore also $c^{k}_{h} : \Omega \to S_{h}$ are $\mathcal{F}(t_{k})$-measurable maps.

\subsection{Discrete a priori estimate}

In a first step towards an error analysis we derive natural a-priori estimates for the solution to the discrete problem.

\begin{lemma}[Discrete a-priori estimates]\label{dis-apriori-g}
Let $c_{h}^{k}$, $k=0,1, \ldots$ be computed according to Algorithm~\ref{fullydiscrete-scheme-g}.
Moreover let $\triangle t \leq \triangle t_{0} $ be sufficiently small.
Then we have that
\begin{multline*}
\mathbb{E}\Big{(}\| c_{h}^{m} \sqrt{|u^{m}_{hx}|}\|^{2}_{H}\Big{)}
+ \sum_{k=1}^{m} \mathbb{E} \Big{(}\| (c_{h}^{k}-c_{h}^{k-1}) \sqrt{|u^{k-1}_{hx}|}\|^{2}_{H}\Big{)}
+ \triangle t \sum_{k=1}^{m} \mathbb{E} \left ( \left \| \frac{ c^{k}_{hx}}{\sqrt{|u^{k}_{hx}|}} \right \|^{2}_{H} \right ) \\
\leq C \Big{(}1 + \mathbb{E}\Big{(}\| c_{h}^{0}\sqrt{ |u^{0}_{hx}|}\|^{2}_{H}\Big{)} \Big{)}
\end{multline*}
holds for any $m =1, \ldots, M$.
\end{lemma}

\begin{proof}
Fixing $\omega \in \Omega$ in \eqref{eq-3.6-g} and taking $\varphi_{h}= c_{h}^{k}(\omega)$, we obtain using
the elementary equality (for $c^{j}$, $a^{j} \in \R$)
\begin{align}\label{algebra}
(c^{k}a^{k} - c^{k-1}a^{k-1})c^{k}= \frac{1}{2} [(c^{k})^{2}a^{k} -(c^{k-1})^{2}a^{k-1}] 
+ \frac{1}{2} (c^{k}-c^{k-1})^{2}a^{k-1} + \frac{1}{2} (c^{k})^{2} (a^{k}-a^{k-1})
\end{align}
that
\begin{align}\label{InMezzo}
\frac{1}{2 } &\| c^{k}_{h}\sqrt{|u^{k}_{hx}|} \|^{2}_{H}- \frac{1}{2 } \| c_{h}^{k-1}\sqrt{|u^{k-1}_{hx}|}\|^{2}_{H} +\frac{1}{2 }\|(c_{h}^{k} -c_{h}^{k-1})\sqrt{|u^{k-1}_{hx}|}\|^{2}_{H}
+\triangle t \left \| \frac{c_{hx}^{k}}{\sqrt{|u^{k}_{hx}| } } \right \|^{2}_{H} \notag \\
&= -\frac{1}{2} \int_{S^{1}}|c_{h}^{k}|^{2} (|u^{k}_{hx}|-|u^{k-1}_{hx}|) dx
-\triangle t \langle c^{k-1}_{h}, c^{k}_{hx} w_{Th}^{k} \rangle_{H} \notag \\
& \quad 
+ \langle \Bapprox(c^{k-1}_{h})\triangle W_{k},c^{k}_{h} -c^{k-1}_{h} \rangle_{H}
+ \langle \Bapprox(c^{k-1}_{h})\triangle W_{k}, c^{k-1}_{h}\rangle_{H} \notag \\
& =I+II+III+IV.
\end{align}

Using that for a $C^{1}$ map $v$ there holds $\|(I_{h}v )_{x}\|_{L^{\infty}} \leq \| v_{x}\|_{L^{\infty}}$, and using the linearity of the interpolation operator, we see that
\begin{align}\label{err-uhx}
||u^{k}_{hx}|-|u^{k-1}_{hx}|| &\leq |u^{k}_{hx} - u^{k-1}_{hx} | \notag \\
&=|(I_{h}[ u(t_{k}, \cdot) - u(t_{k-1}, \cdot)] )_{x}| \leq C\| ( u(t_{k}, \cdot) - u(t_{k-1}, \cdot))_{x} \|_{L^{\infty}} \notag \\
&=C\| u_{x}(t_{k}, \cdot) - u_{x}(t_{k-1}, \cdot) \|_{L^{\infty}} \leq C \triangle t \qquad \text{ with } C =C(\sup_{S^{1} \times [0,T]} |u_{xt}|).
\end{align}
Using \eqref{regularityh} we infer that
\begin{align*}
I =-\frac{1}{2} \int_{S^{1}}|c_{h}^{k}|^{2} (|u^{k}_{hx}|-|u^{k-1}_{hx}|) dx
 \leq C ( \triangle t)\| c^{k}_{h}\sqrt{|u^{k}_{hx}|} \|^{2}_{H}.
\end{align*}
Moreover, note that using \eqref{regularityh}, \eqref{bound-wth}, and a Young inequality we can write
\begin{align*}
II=-\triangle t \langle c^{k-1}_{h}, c^{k}_{hx} w_{Th}^{k} \rangle_{H} \leq
\epsilon \triangle t \left \| \frac{c_{hx}^{k}}{\sqrt{|u^{k}_{hx}| } } \right \|^{2}_{H}
+ C_{\epsilon} \triangle t \| c_{h}^{k-1}\sqrt{|u^{k-1}_{hx}|}\|^{2}_{H}.
\end{align*}

Next, we write
\begin{align*}
III= \langle \Bapprox(c^{k-1}_{h})\triangle W_{k},c^{k}_{h} -c^{k-1}_{h} \rangle_{H}=
& \leq C_{\delta} \| \Bapprox(c^{k-1}_{h})\triangle W_{k}\|^{2}_{H} +\delta \| (c^{k}_{h} - c^{k-1}_{h}) \sqrt{|u^{k-1}_{hx}|}\|_{H}^{2}
\end{align*}
where we have used a Young inequality and \eqref{regularityh}.

Observe that, using It\^o-isometry
\cite[Rem. B.0.6 (iii)]{RoecknerBuch} and \eqref{BH3} we obtain that
\begin{multline*} \mathbb{E} ( \| \Bapprox(c^{k-1}_{h})\triangle W_{k}\|^{2}_{H}) = \mathbb{E} ( \| \int_{t_{k-1}}^{t_{k}} \Bapprox(c^{k-1}_{h}) dW(t')\|^{2}_{H} ) = \int_{t_{k-1}}^{t_{k}} \mathbb{E} (\| B(c^{k-1}_{h}) \proj_{L} \|^{2}_{L_{2}(U,H)}) dt \\
\leq \triangle t \, \mathbb{E} (\| B(c^{k-1}_{h}) \|^{2}_{L_{2}(U,H)}) \leq C(\triangle t) (1 + \mathbb{E}(\|c^{k-1}_{h}\|^{2}_{H} )).
\end{multline*}
Using the stochastic independence
of $\Bapprox(c^{k-1}_{h})$ and $\triangle W_{k}$ we then note that
\begin{align*}
\mathbb{E}(IV)=\mathbb{E} (\langle \Bapprox(c^{k-1}_{h})\triangle W_{k}, c^{k-1}_{h} \rangle_{H})=
\mathbb{E} (\langle \int_{t_{k-1}}^{t_{k}} \Bapprox(c^{k-1}_{h}) dW(t'), c^{k-1}_{h} \rangle_{H} )=
0.
\end{align*}
Therefore choosing $\delta=1/4$, $\epsilon=1/2$, summing up over $k$ between one and $m$, and integrating over $\Omega$ we obtain from \eqref{InMezzo} and \eqref{regularityh} that
\begin{align*}
\frac{1}{2 } \mathbb{E} (\| c^{m}_{h}\sqrt{|u^{}_{hx}|} \|^{2}_{H})- \frac{1}{2 } \mathbb{E}( \| c_{h}^{0}\sqrt{|u^{0}_{hx}|}\|^{2}_{H} ) \, &+ \sum_{k=1}^{m}\frac{1}{4 } \mathbb{E}(\|(c_{h}^{k} -c_{h}^{k-1})\sqrt{|u^{k-1}_{hx}|}\|^{2}_{H})\\
+\sum_{k=1}^{m} \frac{1}{2}\triangle t \mathbb{E} \left(\left \| \frac{c_{hx}^{k}}{\sqrt{|u^{k}_{hx}| } } \right \|^{2}_{H} \right)
&\leq \sum_{k=1}^{m} C(\triangle t) (1 + \mathbb{E}(\|c^{k-1}_{h} \sqrt{|u^{k-1}_{hx}|}\|^{2}_{H} )) \\
& \quad + \sum_{k=1}^{m} C ( \triangle t) \mathbb{E}(\| c^{k}_{h}\sqrt{|u^{k}_{hx}|} \|^{2}_{H}).
\end{align*}
For $ \triangle t $ sufficiently small we can absorb one term on the right-hand side to obtain (neglecting the positive sums and using $m \triangle t \leq T$)
\begin{align*}
\mathbb{E} (\| c^{m}_{h}\sqrt{|u^{}_{hx}|} \|^{2}_{H}) \leq C \mathbb{E}( \| c_{h}^{0}\sqrt{|u^{0}_{hx}|}\|^{2}_{H} ) +CT + \sum_{k=1}^{m-1} C(\triangle t ) \mathbb{E}(\|c^{k}_{h} \sqrt{|u^{k}_{hx}|}\|^{2}_{H} ).
\end{align*}
Application of the discrete Gronwall Lemma \ref{discreteGronwall} together with $m \triangle t \leq M \triangle t= T$ yields the claim.
\end{proof}

\begin{lemma}[Discrete Gronwall]\label{discreteGronwall}
Assume that numbers $\Phi_i$ satisfy
$ 
0 \leq \Phi_i \leq \sum_{j=1}^{i-1} A_j \Phi_j+C
$ 
for $i=1, \ldots, s$, where $A_j,C \geq 0$. Then $\Phi_i \leq C \exp ( \sum_{j=1}^{i-1} A_j ) $, $(i=1,\ldots,s)$.
\end{lemma}

\subsection{Ritz projection}

For the error analysis a suitable Ritz projection is required. To clarify its geometric meaning it is formulated in terms of integrals on the evolving curve rather than on the reference domain. For this purpose, we write $ds(t) = |u_{x}(t)| dx$ for the length element and $\partial_{s} = \frac{1}{|u_{x}(t)|} \partial_{x}$ for the arc-length derivative. We also denote with $\langle \cdot , \cdot \rangle_{L^{2}(ds(t))}$ the $L^2$ inner product of functions on the curve. Note that thanks to \eqref{regularity} the corresponding $L^2$ norm is equivalent to the $L^2$ norm for functions on $S^1$ given in \eqref{VH}. Moreover, $\langle f, g_{x} \rangle_{H} = \langle f, \partial_{s} g \rangle_{L^2(ds(t))}$ for any $f \in H$ and $g \in V$.

\begin{defi}[Geometrical Ritz Projection]
For $z \in V$, $t \in [0,T]$, with the usual notation $ds= |u_{x}(t)| dx$, let $\mathcal{R}_{h}(z)\in S_{h} $ be such that
$$ \int_{S^{1}} z \,ds= \int_{S^{1}}\mathcal{R}_{h}z \,ds $$
and
$$ \langle \partial_{s} z, \partial_{s} \varphi_{h}\rangle_{L^{2}(ds(t))}= \int_{S^{1}} \partial_{s} z \, \partial_{s} \varphi_{h} ds = \int_{S^{1}} \partial_{s} (\mathcal{R}_{h} z ) \, \partial_{s} \varphi_{h} ds = \langle \partial_{s} (\mathcal{R}_{h} z ), \partial_{s} \varphi_{h}\rangle_{L^{2}(ds(t))} \quad \forall \, \varphi_{h} \in S_{h}.$$
\end{defi}

Using standard arguments and \eqref{regularity} it is not difficult to show (for instance, see \cite{DE12}) that the Ritz projection exists, is unique, and satisfies the estimates
\begin{align} \label{Ritz1}
 \| (z - \mathcal{R}_{h} (z) )_{x}\|_{H} \leq C \| \partial_{s} z - \partial_{s} \mathcal{R}_{h} (z) \|_{L^{2}(ds(t))} &\leq C h^{s-1} \|z\|_{W^{s,2}(S^{1})} \qquad s=1,2,\\
\label{Ritz2}
 \| z - \mathcal{ R}_{h} (z) \|_{H} \leq C \| z - \mathcal{R}_{h} (z) \|_{L^{2}(ds(t))} &\leq C h^{s} \|z\|_{W^{s,2}(S^{1})} \qquad s=1,2.
\end{align}

Technically, $\mathcal{R}_{h}$ depends on $|u_{x}(t)|$ and thus on $t$, but this dependence is dropped for convenience. However, it is important to bear this in mind when the Ritz projection is applied to $t$-dependent families $z(t) \subset V$. For this reasons, the Ritz projection is no linear operator: $\mathcal{R}_{h}(z(t) \pm z(\tau)) \neq \mathcal{R}_{h}(z(t)) \pm \mathcal{R}_{h}(z(\tau))$ for functions $z(t) \in V$, $t \in [0,T]$.

It is therefore important to determine how close $\mathcal{R}_{h}$ is to a linear operator and how well $\mathcal{R}_{h}$ preserves the H\"older properties from Assumption~\ref{lass-cholder}. To that end we give the following lemma.

\begin{lemma}[Properties of $\mathcal{R}_{h}$]
Let $z(t) \in W^{s,2}(S^{1})$ for any $t \in [0,T]$ and for $s \in \{ 1,2 \}$. Then, for any $t, \tau \in [0,T]$ we have
\begin{align}
\label{RH1}
 \| (\mathcal{R}_{h} (z(t)) -\mathcal{R}_{h} (z(\tau)) )_{x} \|_{H}& \leq C_{g} |t -\tau| h^{s-1}\| z(\tau)\|_{W^{s,2}(S^{1})} + C \| (z(t) -z(\tau) )_{x} \|_{H},\\
 \label{RH2}
 \| \mathcal{R}_{h} (z(t)) -\mathcal{R}_{h} (z(\tau)) \|_{H}& \leq C_{g} |t -\tau| \| z(\tau)\|_{V}+
 C\| z(t) -z(\tau) \|_{V}.
\end{align}
More precisely, we have that
\begin{multline}\label{RH3}
\| (\mathcal{R}_{h} (z(t)) -\mathcal{R}_{h} (z(\tau)) - (z(t) -z(\tau)) )_{x}\|_{H} \\
\leq Ch^{s-1}
\| z(t) -z(\tau)\|_{W^{s,2}(S^{1})} + C_{g} |t -\tau|h^{s-1}
\| z(\tau)\|_{W^{s,2}(S^{1})} ,
\end{multline}
and 
\begin{multline}\label{RH4}
\| \mathcal{R}_{h} (z(t)) -\mathcal{R}_{h} (z(\tau)) - (z(t) -z(\tau)) \|_{H} \\
 \leq C_{g}|\tau-t|h^{s}\| z(\tau)\|_{W^{s,2}(S^{1})} + Ch^{s}
\| z(t) -z(\tau)\|_{W^{s,2}(S^{1})} .
\end{multline}
The constant $C_{g}$ depends on the evolution of the geometry.
In particular, if the velocity $u_{x}$ does not change in time (i.e. the curve is stationary) then $C_{g}=0$.
\end{lemma}

\begin{proof}
By definition for any time $t \in [0,T]$ and $\varphi_{h} \in S_{h}$ we have that
$$\int_{S^{1}} (z(t))_{x} \frac{\varphi_{hx}}{|u_{x}(t)|} dx = \int_{S^{1}} ( \mathcal{R}_{h} z(t))_{x} \frac{\varphi_{hx}}{|u_{x}(t)|} dx $$
therefore, for any $t, \tau \in [0,T]$ we have
\begin{align}\label{start}
\int_{S^{1}}& (\mathcal{R}_{h} (z(t)) -\mathcal{R}_{h} (z(\tau)) )_{x} \frac{\varphi_{hx}}{|u_{x}(t)|}dx \\
&= \int_{S^{1}} (z(t) -\mathcal{R}_{h} (z(\tau)) )_{x} \varphi_{hx} \left(\frac{1}{|u_{x}(t)|} - \frac{1}{|u_{x}(\tau)|}\right)dx
+ \int_{S^{1}} (z(t) -z(\tau) )_{x} \frac{\varphi_{hx}}{|u_{x}(\tau)|}dx . \notag
\end{align}
Choosing $\varphi_{h}=\mathcal{R}_{h} (z(t)) -\mathcal{R}_{h} (z(\tau)) $ and using \eqref{regularity} and the regularity of $u$ we obtain
\begin{align*}
\| (\mathcal{R}_{h} (z(t)) -\mathcal{R}_{h} (z(\tau)) )_{x} \|_{H}^{2} &\leq
\int_{S^{1}} |(\mathcal{R}_{h} (z(t)) -\mathcal{R}_{h} (z(\tau)) )_{x}|^{2} \frac{1}{|u_{x}(t)|}dx \\
&\leq C |t -\tau|\| ( z(t) -z(\tau) + z(\tau) -\mathcal{R}_{h} (z(\tau)) )_{x} \|_{H} \| (\mathcal{R}_{h} (z(t)) -\mathcal{R}_{h} (z(\tau)) )_{x} \|_{H}\\
& \quad + C \| (z(t) -z(\tau) )_{x} \|_{H} \| (\mathcal{R}_{h} (z(t)) -\mathcal{R}_{h} (z(\tau)) )_{x} \|_{H}.
\end{align*}
Inequality \eqref{RH1} follows from \eqref{Ritz1}.

Next, we use a Poincar\'{e}-inequality for a map $\funch:S^{1} \to \R$ along the curve $u(t)$:
$$\| \funch - \bar{h}\|_{L^{2}(ds(t))} \leq C\| \partial_{s} \funch \|_{L^{2}(ds(t))} \qquad \text{ where } \bar{h}:= \frac{1}{|S^{1}|}\int_{S^{1}} \funch ds(t)= \frac{1}{|S^{1}|}\int_{S^{1}} \funch |u_{x}(t)| dx.$$
Note that $C$ depends on the length of the curve $u(t)$, which is uniformly bounded in $t$ by \eqref{regularity}. We obtain, again with the help of \eqref{regularity}, that
\begin{align*}
\|\funch\|_{H} \leq C \| \funch\|_{L^{2}(ds(t))} \leq C\| \funch - \bar{h}\|_{L^{2}(ds(t))} + C\| \bar{h}\|_{L^{2}(ds(t))} \leq C \| \funch_{x}\|_{H} + C |\bar{h}|.
\end{align*}
Choosing $\funch:=\mathcal{R}_{h} (z(t)) -\mathcal{R}_{h} (z(\tau)) $ and using \eqref{RH1} we obtain \eqref{RH2} since
\begin{align*}
|\bar{h}| &\leq C \Big|\int_{S^{1}} \mathcal{R}_{h} (z(t)) |u_{x}(t)| - \mathcal{R}_{h} (z(\tau)) |u_{x}(t)| dx \Big|\\
&\leq C \Big|\int_{S^{1}} \mathcal{R}_{h} (z(t)) |u_{x}(t)| - \mathcal{R}_{h} (z(\tau)) |u_{x}(\tau)| dx \Big| + C \Big|\int_{S^{1}} \mathcal{R}_{h} (z(\tau)) ( |u_{x}(\tau)| - |u_{x}(t)|) dx \Big|\\
& \leq C \Big|\int_{S^{1}} z(t) |u_{x}(t)| - z(\tau) |u_{x}(\tau)| dx \Big| + C |t -\tau| \| \mathcal{R}_{h} (z(\tau)) \|_{L^{1}(S^{1})}\\
& \leq C \| z(t) -z(\tau)\|_{H} + C |t -\tau| ( \| z(\tau)\|_{H} +\| \mathcal{R}_{h} (z(\tau)) \|_{H} )\\
&\leq C \| z(t) -z(\tau)\|_{H} + C |t -\tau| ( \| z(\tau)\|_{H} +\| \mathcal{R}_{h} (z(\tau)) -z(\tau)\|_{H} )\\
& \leq C \| z(t) -z(\tau)\|_{H} + C |t -\tau| (\| z(\tau)\|_{H} +h^{s}\| z(\tau)\|_{W^{s,2}(S^{1})})
\end{align*}
where we have used \eqref{Ritz2} in the last inequality.

Next, starting again from \eqref{start} we obtain for $\varphi_{h} \in S_{h}$ 
\begin{align*}
\int_{S^{1}}& [(\mathcal{R}_{h} (z(t)) -\mathcal{R}_{h} (z(\tau))) -(z(t) -z(\tau)) ]_{x} \frac{\varphi_{hx} }{|u_{x}(t)|}dx \\
&= \int_{S^{1}} (z(t) -\mathcal{R}_{h} (z(\tau)) )_{x} \varphi_{hx} \left(\frac{1}{|u_{x}(t)|} - \frac{1}{|u_{x}(\tau)|}\right)dx \\
& \quad
+ \int_{S^{1}} (z(t) -z(\tau) )_{x} \varphi_{hx} \left( \frac{1}{|u_{x}(\tau)|} - \frac{1}{|u_{x}(t)|} \right)dx .
\end{align*}
Choosing $\varphi_{h}=(\mathcal{R}_{h} (z(t)) -\mathcal{R}_{h} (z(\tau))) -I_{h}(z(t) -z(\tau))=
(\mathcal{R}_{h} (z(t)) -\mathcal{R}_{h} (z(\tau))) - (z(t) -z(\tau))+(z(t) -z(\tau)) -I_{h}(z(t) -z(\tau))$
we obtain, setting for simplicity of exposition
\begin{align*}
R_{h}:=\mathcal{R}_{h} (z(t)) -\mathcal{R}_{h} (z(\tau)), \qquad Z:=z(t) -z(\tau), \qquad I_{h}Z:= I_{h}(z(t) -z(\tau))=I_{h}z(t) -I_{h}z(\tau),
\end{align*}
and using \eqref{regularity}
\begin{align*}
d_{0} \|(R_{h}-Z)_{x}\|_{H}^{2}
&\leq
\int_{S^{1}}\frac{|(R_{h}-Z)_{x}|^{2}}{|u_{x}(t)|} dx=\int_{S^{1}}\frac{(R_{h}-Z)_{x} (I_{h}Z-Z)_{x}}{|u_{x}(t)|}\\
& \quad 
 + \int_{S^{1}} (z(\tau) -\mathcal{R}_{h} (z(\tau)) )_{x} (R_{h}-I_{h}Z)_{x}\left(\frac{1}{|u_{x}(t)|} - \frac{1}{|u_{x}(\tau)|}\right)dx \\
& \leq \epsilon \|(R_{h}-Z)_{x}\|_{H}^{2} + C_{\epsilon} \|(I_{h}Z-Z)_{x}\|_{H}^{2} \\
& \quad + C |t -\tau|
\| (z(\tau) -\mathcal{R}_{h}(z(\tau)))_{x}\|_{H} ( \|(R_{h}-Z)_{x}\|_{H} + \|(I_{h}Z-Z)_{x}\|_{H})\\
& \leq \epsilon \|(R_{h}-Z)_{x}\|_{H}^{2} + C_{\epsilon} h^{2(s-1)}\|Z\|_{W^{s,2}(S^{1})}^{2} \\
& \quad + C |t -\tau|
h^{s-1}\| z(\tau)\|_{W^{s,2}(S^{1})} ( \|(R_{h}-Z)_{x}\|_{H} + h^{s-1}\|Z\|_{W^{s,2}(S^{1})})
\end{align*}
where we have used \eqref{Ritz1}, \eqref{(4.2)extra} and \eqref{(4.2)bis} in the last inequality.
This yields
\begin{align*}
\|(R_{h}-Z)_{x}\|_{H} \leq C h^{s-1}\|Z\|_{W^{s,2}(S^{1})} + C |t -\tau|
h^{s-1}\| z(\tau)\|_{W^{s,2}(S^{1})} 
\end{align*}
which in turn gives \eqref{RH3}. Next we want to estimate the $L^{2}$-norm of $R_{h}-Z$ and ``win'' a factor $h$. To that end an Aubin-Nitsche trick must be employed.
With $ds:=ds(t)=|u_{x}(t)|dx$ and using Lax-Milgram we (weakly) solve the problem : find $w:S^{1} \to \R$ such that
\begin{align*}
\partial^{2}_{s} w = R_{h}-Z -\bar{m}, \qquad \int_{S^{1}}w ds=0,
\end{align*}
where (using the definition of the Ritz projection)
\begin{align*} \bar{m}:&=\int_{S^{1}}(R_{h}-Z) ds=\int_{S^{1}}(R_{h}-Z) |u_{x}(t)|dx= \int_{S^{1}} (z(\tau) -\mathcal{R}_{h}(z(\tau)))|u_{x}(t)|dx\\
&= \int_{S^{1}} (z(\tau) -\mathcal{R}_{h}(z(\tau)))(|u_{x}(t)|- |u_{x}(\tau)|)dx .
\end{align*}
Note that using the regularity of $u$ and \eqref{Ritz2} we have
\begin{align}\label{mbarest}
|\bar{m}| \leq C |t-\tau|\|z(\tau) -\mathcal{R}_{h}(z(\tau))\|_{H} \leq C |t-\tau| h^{s}\|z(\tau)\|_{W^{s,2}(S^{1})}.
\end{align}
Moreover, the Poincar\'{e}-inequality (on curves) gives $\| w\|_{L^{2}(ds(t))} \leq C\|\partial_{s} w\|_{L^{2}(ds(t))}$ (with a uniform constant $C$ independent of $t$ thanks to \eqref{regularity})
so that from the weak formulation we immediately infer
$\|\partial_{s} w\|_{L^{2}(ds(t))} \leq C \| R_{h}-Z -\bar{m} \|_{L^{2}(ds(t))} $.
Regularity theory finally yields
\begin{align*}
\| w \|_{W^{2,2}(ds(t))} \leq C \| R_{h}-Z -\bar{m} \|_{L^{2}(ds(t))}
\end{align*}
that is (using \eqref{regularity})
\begin{align}\label{reg-ww}
\| w \|_{W^{2,2}(S^{1})} \leq C \| R_{h}-Z -\bar{m} \|_{L^{2}(S^{1})}.
\end{align}
Exploiting the weak formulation for the solution $w$ we can write
\begin{align*}
\| R_{h}-Z -\bar{m}\|_{L^{2}(ds(t))}^{2} &=\int_{S^{1}} \partial_{s} w \partial_{s} (R_{h}-Z) ds
\\
&=\int_{S^{1}} \partial_{s} (w -w_{h})\partial_{s} (R_{h}-Z) ds + \int_{S^{1}} \partial_{s} w_{h}\partial_{s} (R_{h}-Z) ds=I+II
\end{align*}
where $w_{h} \in S_{h}$ will be chosen later. Using the definition of Ritz operator (and recalling that $ds=ds(t)$) we have that
\begin{align*}
II:&=\int_{S^{1}} \partial_{s} w_{h}\partial_{s} (R_{h}-Z) ds=\int_{S^{1}} \partial_{s} w_{h}\partial_{s} (z(\tau)-\mathcal{R}_{h}(z(\tau))) ds\\
&=\int_{S^{1}} w_{hx}\frac{(z(\tau)-\mathcal{R}_{h}(z(\tau)))_{x}}{|u_{x}(t)|} dx \\
&=\int_{S^{1}} w_{hx}\,(z(\tau)-\mathcal{R}_{h}(z(\tau)))_{x} \Big(\frac{1}{|u_{x}(t)|} - \frac{1}{|u_{x}(\tau)|} \Big) dx\\
&=\int_{S^{1}} (w_{hx}-w_{x})\,(z(\tau)-\mathcal{R}_{h}(z(\tau)))_{x} \Big(\frac{1}{|u_{x}(t)|} - \frac{1}{|u_{x}(\tau)|} \Big) dx \\
& \quad +
\int_{S^{1}} w_{x}\,(z(\tau)-\mathcal{R}_{h}(z(\tau)))_{x} \Big(\frac{1}{|u_{x}(t)|} - \frac{1}{|u_{x}(\tau)|} \Big) dx\\
& =
\int_{S^{1}} (w_{hx}-w_{x})\,(z(\tau)-\mathcal{R}_{h}(z(\tau)))_{x} \Big(\frac{1}{|u_{x}(t)|} - \frac{1}{|u_{x}(\tau)|} \Big) dx \\
&\quad -\int_{S^{1}} w_{xx}\,(z(\tau)-\mathcal{R}_{h}(z(\tau))) \Big(\frac{1}{|u_{x}(t)|} - \frac{1}{|u_{x}(\tau)|} \Big) dx \\
& \quad- \int_{S^{1}} w_{x}\,(z(\tau)-\mathcal{R}_{h}(z(\tau))) \Big(\frac{1}{|u_{x}(t)|} - \frac{1}{|u_{x}(\tau)|} \Big)_{x} dx.
\end{align*}
Therefore, using the regularity of $u$, \eqref{Ritz1}, \eqref{Ritz2}, we obtain
\begin{align*}
II &\leq C|\tau-t|h^{s-1}\| z(\tau)\|_{W^{s,2}(S^{1})} \|( w_{h} -w)_{x}\|_{L^{2}(S^{1})} \\
& \quad + C|\tau-t|h^{s}\| z(\tau)\|_{W^{s,2}(S^{1})} (\|w_{x}\|_{L^{2}(S^{1})} + \|w_{xx}\|_{L^{2}(S^{1})}).
\end{align*}
Choosing $w_{h}=I_{h}w$ and using \eqref{(4.2)}, \eqref{reg-ww}, \eqref{regularity}, we infer
\begin{align*}
II &\leq C|\tau-t|h^{s}\| z(\tau)\|_{W^{s,2}(S^{1})} \| R_{h}-Z -\bar{m} \|_{L^{2}(S^{1})},\\
I & \leq Ch\| R_{h}-Z -\bar{m} \|_{L^{2}(S^{1})} \| (R_{h}-Z)_{x}\|_{L^{2}(S^{1})}.
\end{align*}
Therefore
\begin{align*}
\| R_{h}-Z -\bar{m} \|_{L^{2}(S^{1})} \leq C|\tau-t|h^{s}\| z(\tau)\|_{W^{s,2}(S^{1})} + Ch
\| (R_{h}-Z)_{x}\|_{L^{2}(S^{1})}
\end{align*}
and then, using \eqref{mbarest},
\begin{align*}
\| R_{h}-Z \|_{L^{2}(S^{1})} \leq C|\tau-t|h^{s}\| z(\tau)\|_{W^{s,2}(S^{1})} + Ch
\| (R_{h}-Z)_{x}\|_{L^{2}(S^{1})}. 
\end{align*}
Together with \eqref{RH3} this yields \eqref{RH4}.
\end{proof}

\subsection{Error estimates}

The procedure to estimate the error and to quantify convergence is similar to the proof of the a-priori bounds above.
Analogous ideas can be found also in the proof of \cite[Theorem 4.4]{BDE}, where a domain decomposition approach is analyzed: there however only time-discretization is taken into account.

For later reference let us recall that a solution from Definition~\ref{def_solution} satisfies the equation
\begin{align}\label{eq4c}
\langle c(t_{k}) |u_{x}(t_{k})|, \varphi \rangle_{H} &- \langle c(t_{k-1}) |u_{x}(t_{k-1})|, \varphi \rangle_{H}
 = -\int_{t_{k-1}}^{t_{k}} \langle \partial_{s} c(t'), \partial_{s} \varphi \rangle_{L^{2} (ds(t'))} dt' \notag \\
&-\int_{t_{k-1}}^{t_{k}}
\langle w_{T}(t') c(t'), \partial_{s} \varphi \rangle_{L^{2} (ds(t'))} dt'
+ \sum_{l=1}^{\infty}\int_{t_{k-1}}^{t_{k}} \langle B (c(t')) g_{l} , \varphi \rangle_{H} d \beta_{l}(t')
\end{align}
for all $\varphi \in V$.

\begin{teo}[Error estimates] \label{err-est-g}
Let $c$ be a solution according to Definition~\ref{def_solution} for some initial data $c(0)=c_{0} \in V$, and let Assumption~\ref{lass-cholder} and Assumption~\ref{approximW} hold. Let $c_{h}^{k}$ be computed according to Algorithm~\ref{fullydiscrete-scheme-g}.
Further let $h \leq h_{0}$ and $\triangle t \leq \triangle t_{0}$ sufficiently small.
Let $\triangle t=h$. 
Then the following error estimate holds for any $k=1, \ldots, M$:
\begin{align*}
\mathbb{E}(\|c^{k}_{h}- c(t_{k})\|_{H}^{2})^{\frac{1}{2}}\leq C \sqrt{\epsilon_{W}} + C (\triangle t)^{\frac{\nu_{r}}{2}} + C \mathbb{E}( \|c^{0}_{h}- c_{0}\|_{H}^{2})^{\frac{1}{2}}+Ch.
\end{align*}
\end{teo}
A comment on the coupling between the time and spatial step sizes is provided after the proof.
\begin{proof}
Setting $$e^{k}:= c^{k}_{h} - \mathcal{R}_{h}(c(t_{k}))$$
i.e. $ e^{k}(\omega, \cdot) = c^{k}_{h}(\omega,\cdot) - \mathcal{R}_{h}(c(\omega,t_{k}, \cdot)) \in S_{h}$,
we obtain using \eqref{eq-3.6-g} and \eqref{eq4c} (and for a fixed $\omega \in \Omega$)
\begin{align*}
\langle e^{k} & |u^{k}_{hx}|, \varphi_{h } \rangle_{H} - \langle e^{k-1} |u^{k-1}_{hx}|, \varphi_{h } \rangle_{H} + \triangle t \langle \frac{e^{k}_{x}}{ |u^{k}_{hx}|}, \varphi_{h } \rangle_{H} \\
& = \left[ \langle c^{k}_{h}| u^{k}_{hx}|, \varphi_{h}\rangle_{H} + \triangle t \langle \frac{c^{k}_{hx}}{| u^{k}_{hx}|}, \varphi_{hx} \rangle_{H} - \langle c^{k-1}_{h} | u^{k-1}_{hx}|, \varphi_{h}\rangle_{H}
\right]\\
&\quad - \left[ \langle \mathcal{R}_{h}(c(t_{k}))| u^{k}_{hx}|, \varphi_{h}\rangle_{H} + \triangle t \langle \frac{(\mathcal{R}_{h}(c(t_{k})))_{x}}{| u^{k}_{hx}|}, \varphi_{hx} \rangle_{H} - \langle \mathcal{R}_{h}(c(t_{k-1})) | u^{k-1}_{hx}|, \varphi_{h}\rangle_{H} \right]
\\
&=\left[-\triangle t \langle c^{k-1}_{h}, \varphi_{hx} w_{Th}^{k} \rangle_{H} + \langle \Bapprox(c^{k-1}_{h})\triangle W_{k}, \varphi_{h}\rangle_{H} \right]
\\
& \quad - \big[ \langle \mathcal{R}_{h}(c(t_{k})) | u^{k}_{hx}| - \mathcal{R}_{h}(c(t_{k-1})) | u^{k-1}_{hx}|, \varphi_{h}\rangle_{H} + \triangle t \langle (\mathcal{R}_{h}(c(t_{k})))_{x} ( \frac{1}{| u^{k}_{hx}|} - \frac{1}{| u_{x}(t_{k})|}), \varphi_{hx} \rangle_{H} \big]\\
& \quad - \triangle t \langle (\mathcal{R}_{h}(c(t_{k})))_{x} \frac{1}{| u_{x}(t_{k})|}, \varphi_{hx} \rangle_{H} + \langle c(t_{k}) |u_{x}(t_{k})|-c(t_{k-1}) |u_{x}(t_{k-1})|, \varphi_{h}\rangle_{H} \\
&\quad - 
\big[ \langle c(t_{k}) |u_{x}(t_{k})|, \varphi_{h}\rangle_{H}
- \langle c(t_{k-1}) |u_{x}(t_{k-1})|, \varphi_{h}\rangle_{H} \big]\\
&=\left[-\triangle t \langle c^{k-1}_{h}, \varphi_{hx} w_{Th}^{k} \rangle_{H}
+ \int_{t_{k-1}}^{t_{k}} \langle w_{T} c(t'), \partial_{s} \varphi_{h} \rangle_{L^{2} (ds(t'))} dt' \right]\\
& \quad
+ \left[ \langle \Bapprox(c^{k-1}_{h})\triangle W_{k}, \varphi_{h}\rangle_{H}
-\sum_{l=1}^{\infty}\int_{t_{k-1}}^{t_{k}} \langle B (c(t')) g_{l} , \varphi_{h} \rangle_{H} d \beta_{l}(t') 
\right]
\\
& \quad +\left [ \int_{t_{k-1}}^{t_{k}} \langle \partial_{s} c(t'), \partial_{s} \varphi_{h} \rangle_{L^{2}(ds(t'))} dt' -\triangle t \langle (\mathcal{R}_{h}(c(t_{k})))_{x} \frac{1}{| u_{x}(t_{k})|}, \varphi_{hx} \rangle_{H} \right ]
\\
& \quad +\Big[\langle c(t_{k}) |u_{x}(t_{k})| - c(t_{k-1}) |u_{x}(t_{k-1})|, \varphi_{h}\rangle_{H} -\langle \mathcal{R}_{h}(c(t_{k})) | u^{k}_{hx}| -\mathcal{R}_{h}(c(t_{k-1})) | u^{k-1}_{hx}|, \varphi_{h}\rangle_{H}
\Big]
\\
& \quad
- \triangle t \langle (\mathcal{R}_{h}(c(t_{k})))_{x} ( \frac{1}{| u^{k}_{hx}|} - \frac{1}{| u_{x}(t_{k})|}), \varphi_{hx} \rangle_{H} \\
&=I +II+III+IV+V.
\end{align*}
We now choose $\varphi_{h}=e^{k}$ and, with a slight abuse of notation, still write $I, \ldots, V$ for the corresponding terms. Employing \eqref{algebra}
we infer that
\begin{align} \label{starting-eq}
\frac{1}{2 } & \left \| e^{k}\sqrt{|u^{k}_{hx}|} \right \|^{2}_{H}- \frac{1}{2 } \left \| e^{k-1}\sqrt{|u^{k-1}_{hx}|} \right \|^{2}_{H} + \frac{1}{2 } \left \| (e^{k} -e^{k-1})\sqrt{|u^{k-1}_{hx}|} \right \|^{2}_{H}
+\triangle t \left \| \frac{e_{x}^{k}}{\sqrt{|u^{k}_{hx}| } } \right \|^{2}_{H} \notag \\
&= -\frac{1}{2} \int_{S^{1}}|e^{k}|^{2} (|u^{k}_{hx}|-|u^{k-1}_{hx}|) dx + I +II+III+IV+V \notag\\
&=VI + I +II+III+IV+V.
\end{align}
We will now estimate each term, keeping in mind that at the end we will employ a Gronwall argument to the above equation after integration over $\Omega$.

Using \eqref{err-uhx} and \eqref{regularityh} we immediately infer that
\begin{align*}
VI=-\frac{1}{2} \int_{S^{1}}|e^{k}|^{2} (|u^{k}_{hx}|-|u^{k-1}_{hx}|) dx \leq C \triangle t \| e^{k}\sqrt{|u^{k}_{hx}|} \|^{2}_{H}
\end{align*}
so that
\begin{align}\label{term6}
\mathbb{E}(VI) \leq C \triangle t \,\mathbb{E} ( \| e^{k}\sqrt{|u^{k}_{hx}|} \|^{2}_{H}).
\end{align}

For term $I$ we compute
\begin{align*}
I&= -\triangle t \langle c^{k-1}_{h}, e_{x}^{k} w_{Th}^{k} \rangle_{H}
+ \int_{t_{k-1}}^{t_{k}} \langle w_{T} c(t'), \partial_{s} e^{k} \rangle_{L^{2} (ds(t'))} dt' \\
& = -\int_{t_{k-1}}^{t_{k}} \int_{S^{1}} c^{k-1}_{h} e_{x}^{k} w_{Th}^{k} dx dt' + \int_{t_{k-1}}^{t_{k}} \int_{S^{1}} w_{T} (t') c(t') e^{k}_{x} dx dt'\\
&= \int_{t_{k-1}}^{t_{k}} \int_{S^{1}} (w_{T}(t') -w_{Th}^{k}) c(t') e^{k}_{x} dx dt'
+ \int_{t_{k-1}}^{t_{k}} \int_{S^{1}} w_{Th}^{k} (c(t')- \mathcal{R}_{h}(c(t' )))e^{k}_{x} dx dt'\\
& \quad 
+ \int_{t_{k-1}}^{t_{k}} \int_{S^{1}} w_{Th}^{k} (\mathcal{R}_{h}(c(t' ))- \mathcal{R}_{h} (c(t_{k-1})) )e^{k}_{x} dx dt'
+ \int_{t_{k-1}}^{t_{k}} \int_{S^{1}} w_{Th}^{k} (\mathcal{R}_{h} (c(t_{k-1})) - c^{k-1}_{h} )e^{k}_{x} dx dt'\\
&= I_{1}+I_{2}+I_{3}+I_{4}.
\end{align*}
Using that for any $t', t_{*}' \in [t_{k-1}, t_{k}]$ we have
\begin{align*}
\| w_{T}(t') -w_{Th}^{k}\|_{L^{\infty}(S^{1})} &\leq \| w_{T}(t') - w_{T}(t_{k})\|_{L^{\infty}(S^{1})} + \| w_{T}(t_{k}) -I_{h}( w_{T}(t_{k}))\|_{L^{\infty}(S^{1})} \\
& \leq \triangle t \| \partial_{t} w_{T}(t_{*}') \|_{L^{\infty}(S^{1})} + Ch^{2} \| \partial_{xx} w_{T}(t_{k}) \|_{L^{\infty}(S^{1})}
\end{align*}
we obtain using the smoothness of $w_{T}$, \eqref{regularityh}, \eqref{regularity}, and a Young inequality
\begin{align*}
\mathbb{E}(I_{1}) &\leq C \int_{t_{k-1}}^{t_{k}} \mathbb{E} \Big( (\triangle t + h^{2}) \| c(t') \|_{L^{2}(ds(t'))} \big \| \frac{e^{k}_{x}}{\sqrt{|u^{k}_{hx}|}} \big \|_{L^{2}(S^{1})} \Big) dt'\\
&
\leq \epsilon \triangle t \mathbb{E} \left (\left \| \frac{e_{x}^{k}}{\sqrt{|u^{k}_{hx}| } } \right \| ^{2}_{H} \right) + C_{\epsilon}(\triangle t + h^{2})^{2} \int_{t_{k-1}}^{t_{k}}\mathbb{E} (\| c(t')\|^{2}_{V} )dt'.
\end{align*}
Observe that later on we will make use of Lemma~\ref{c-apriori-est} for the last integral term.
Using \eqref{bound-wth},\eqref{regularityh}, \eqref{Ritz2}, and a Young inequality
 we infer
\begin{align*}
\mathbb{E}(I_{2}) &\leq C \int_{t_{k-1}}^{t_{k}} \mathbb{E} \Big( \| c(t')- \mathcal{R}_{h}(c(t'))\|_{L^{2}(S^{1})} \| \frac{e^{k}_{x}}{\sqrt{|u^{k}_{hx}|}}\|_{L^{2}(S^{1})} \Big) dt'\\
& \leq \epsilon \triangle t \mathbb{E} \left (\left \| \frac{e_{x}^{k}}{\sqrt{|u^{k}_{hx}| } } \right \| ^{2}_{H} \right) + C_{\epsilon} h^{2} \int_{t_{k-1}}^{t_{k}}\mathbb{E} (\| c(t')\|^{2}_{V} )dt'.
\end{align*}
Using \eqref{bound-wth},\eqref{regularityh}, a Young inequality, \eqref{RH2} and Assumption~\ref{lass-cholder}
 we infer
\begin{align*}
\mathbb{E}(I_{3}) &\leq C \int_{t_{k-1}}^{t_{k}} \mathbb{E} \Big( \| \mathcal{R}_{h} (c(t'))- \mathcal{R}_{h}( c(t_{k-1}))\|_{L^{2}(S^{1})} \| \frac{e^{k}_{x}}{\sqrt{|u^{k}_{hx}|}}\|_{L^{2}(S^{1})} \Big) dt'\\
& \leq \epsilon \triangle t \, \mathbb{E} \left (\left \| \frac{e_{x}^{k}}{\sqrt{|u^{k}_{hx}| } } \right \| ^{2}_{H} \right) + C_{\epsilon} \int_{t_{k-1}}^{t_{k}} \mathbb{E} (|t' -t_{k-1}|^2 \|c(t')\|^{2}_{V} + \| c(t')-c(t_{k-1})\|^{2}_{V} ) dt'\\
& \leq \epsilon \triangle t \, \mathbb{E} \left (\left \| \frac{e_{x}^{k}}{\sqrt{|u^{k}_{hx}| } } \right \| ^{2}_{H} \right) + C_{\epsilon} \, (\triangle t)^2 \,\int_{t_{k-1}}^{t_{k}} \mathbb{E} ( \|c(t')\|^{2}_{V} ) dt' + C_{\epsilon} (\triangle t)^{1+\nu_{r}}.
\end{align*}
Finally we obtain from \eqref{bound-wth} and \eqref{regularityh} that
\begin{align*}
\mathbb{E}(I_{4}) \leq \epsilon \triangle t \mathbb{E} \left (\left \| \frac{e_{x}^{k}}{\sqrt{|u^{k}_{hx}| } } \right \| ^{2}_{H} \right) + C_{\epsilon} \, \triangle t \, \mathbb{E} (\| e^{k-1}\sqrt{|u^{k-1}_{hx}|}\|_{H}^{2}).
\end{align*}
Collecting all estimates we obtain that
\begin{align}\label{term1}
\mathbb{E}(I) &\leq \epsilon \triangle t \, \mathbb{E} \left (\left \| \frac{e_{x}^{k}}{\sqrt{|u^{k}_{hx}| } } \right \| ^{2}_{H} \right) + C_{\epsilon} \big{(} (\triangle t + h^{2})^2 + h^2 + \triangle t^2 \big{)} \,\int_{t_{k-1}}^{t_{k}} \mathbb{E} ( \|c(t')\|^{2}_{V} ) dt' \\
& \quad + C_{\epsilon} (\triangle t)^{1+\nu_{r}} + C_{\epsilon} \, \triangle t \, \mathbb{E} (\| e^{k-1}\sqrt{|u^{k-1}_{hx}|}\|_{H}^{2}).
\notag
\end{align}

Next, we consider
\begin{align*}
II= \langle \Bapprox(c^{k-1}_{h})\triangle W_{k}, e^{k}\rangle_{H}
-\sum_{l=1}^{\infty}\int_{t_{k-1}}^{t_{k}} \langle B (c(t')) g_{l} , e^{k} \rangle_{H} d \beta_{l}(t') .
\end{align*}
We observe that, using stochastic independence 
and the fact that time steps of Wiener processes have mean zero,
\begin{align*}
\mathbb{E}(\langle \Bapprox(c^{k-1}_{h})\triangle W_{k}, e^{k-1}\rangle_{H} )=0
\end{align*}
Similarly there holds
\begin{align*}
\mathbb{E}\left(\sum_{l=1}^{\infty}\int_{t_{k-1}}^{t_{k}} \langle B (c(t')) g_{l} , e^{k-1} \rangle_{H} d \beta_{l}(t')
 \right)=0.
\end{align*}
Therefore we can write
\begin{align*}
\mathbb{E}(II)&= \mathbb{E} ( \langle \Bapprox(c^{k-1}_{h})\triangle W_{k}, e^{k}-e^{k-1}\rangle_{H} ) - \mathbb{E} (\sum_{l=1}^{\infty}\int_{t_{k-1}}^{t_{k}} \langle B (c(t')) g_{l} , e^{k}-e^{k-1} \rangle_{H} d \beta_{l}(t') )\\
&=\mathbb{E} ( \langle \int_{t_{k-1}}^{t_{k}}B(c^{k-1}_{h}) \proj_{L} \,dW(t'), e^{k}-e^{k-1}\rangle_{H} ) -
\mathbb{E}( \langle \int_{t_{k-1}}^{t_{k}} B (c(t')) dW(t') , e^{k}-e^{k-1} \rangle_{H} )
\\
&=\mathbb{E} ( \langle \int_{t_{k-1}}^{t_{k}}B(c^{k-1}_{h}) \proj_{L} \,dW(t'), e^{k}-e^{k-1}\rangle_{H} )- \mathbb{E} ( \langle \int_{t_{k-1}}^{t_{k}}B(c^{k-1}_{h}) \,dW(t'), e^{k}-e^{k-1}\rangle_{H} )\\
& \quad + \mathbb{E} ( \langle \int_{t_{k-1}}^{t_{k}}B(c^{k-1}_{h}) \,dW(t'), e^{k}-e^{k-1}\rangle_{H} )
- \mathbb{E}( \langle \int_{t_{k-1}}^{t_{k}} B (c(t')) dW(t') , e^{k}-e^{k-1} \rangle_{H} )\\
&=II_{1}+II_{2}.
\end{align*}
We compute using \eqref{regularityh}
\begin{align*}
II_{1}=&\mathbb{E} ( \langle \int_{t_{k-1}}^{t_{k}}B(c^{k-1}_{h}) (\proj_{L}-Id) \,dW(t'), e^{k}-e^{k-1}\rangle_{H} )\\
& \leq C\mathbb{E}(\|( e^{k}-e^{k-1}) \sqrt{|u^{k-1}_{hx}|}\|_{H}^{2})^{\frac{1}{2}}\,
\mathbb{E} \Big( \Big\| \int_{t_{k-1}}^{t_{k}}B(c^{k-1}_{h}) (\proj_{L}-Id) \,dW(t') \Big \|_{H}^{2} \Big)^{\frac{1}{2}}.
\end{align*}
By the It\^o's isometry (see for instance \cite[(10.24)]{Powell}) and Assumption~\ref{approximW} we deduce that
\begin{align*}
\mathbb{E}\left (\Big \| \int_{t_{k-1}}^{t_{k}} ( B(c^{k-1}_{h}) \proj_{L}- B(c^{k-1}_{h}) ) \,dW(t') \Big \|_{H}^{2} \right ) &= \int_{t_{k-1}}^{t_{k}} \mathbb{E} \left(
\| B(c^{k-1}_{h})(\proj_{L} -Id)\|^{2}_{L_{2}(U,H)} \right) dt'\\
& \leq \int_{t_{k-1}}^{t_{k}}\epsilon_{W} (1+ \mathbb{E}(\|c^{k-1}_{h}\|_{H}^{2})) dt' \\
&\leq \triangle t \, \epsilon_{W} (1+ \mathbb{E}(\|c^{k-1}_{h}\|_{H}^{2})).
\end{align*}
Using the a-priori estimates from Lemma~\ref{dis-apriori-g}, \eqref{regularityh}, and then a Young inequality, we obtain that
\begin{align}\label{term-II1}
II_{1}
\leq 
\frac{1}{4} \mathbb{E}\Big(\|( e^{k}-e^{k-1}) \sqrt{|u^{k-1}_{hx}|}\|_{H}^{2} \Big) +
C \, \triangle t \, \epsilon_{W} .
\end{align}
Also for the second term we use \eqref{regularityh} and compute
\begin{align*}
 II_{2} & =\mathbb{E} ( \langle \int_{t_{k-1}}^{t_{k}} (B(c^{k-1}_{h}) - B(c(t'))\,dW(t'), e^{k}-e^{k-1}\rangle_{H} )
\\
 &\leq C\mathbb{E}\Big(\|( e^{k}-e^{k-1}) \sqrt{|u^{k-1}_{hx}|}\|_{H}^{2} \Big)^{\frac{1}{2}} \mathbb{E}\left ( \Big \| \int_{t_{k-1}}^{t_{k}} ( B(c^{k-1}_{h}) - B(c(t')) ) \,dW(t') \Big \|_{H}^{2} \right )^{\frac{1}{2}}\\
& \leq 
\frac{1}{4} \mathbb{E}\Big(\| ( e^{k}-e^{k-1}) \sqrt{|u^{k-1}_{hx}|}\|_{H}^{2} \Big) + 
C \mathbb{E} \left( \Big \| \int_{t_{k-1}}^{t_{k}} ( B(c^{k-1}_{h}) - B(c(t')) ) \,dW(t') \Big \|_{H}^{2} \right ).
\end{align*}
By the It\^o's isometry, \eqref{BH2}, \eqref{Ritz2}, and \eqref{RH2} we deduce that
\begin{align*}
\mathbb{E} &\left( \Big \| \int_{t_{k-1}}^{t_{k}} ( B(c^{k-1}_{h}) - B(c(t')) ) \,dW(t') \Big \|_{H}^{2} \right ) = \int_{t_{k-1}}^{t_{k}} \mathbb{E}( \| B(c^{k-1}_{h}) - B(c(t'))\|^{2}_{L_{2}(U,H)} ) dt' \\
&\leq C \int_{t_{k-1}}^{t_{k}} \mathbb{E}( \| c^{k-1}_{h} - c(t') \|_{H}^{2}) dt'\\
& \leq C \int_{t_{k-1}}^{t_{k}} \mathbb{E}( \| c^{k-1}_{h}-\mathcal{R}_{h}(c(t_{k-1})) \|_{H}^{2} + \|\mathcal{R}_{h}(c(t_{k-1})) - \mathcal{R}_{h}(c(t'))\|_{H}^{2} + \| \mathcal{R}_{h}(c(t')) - c(t') \|_{H}^{2}) dt'\\
& \leq C \, \triangle t \,\mathbb{E}(\| e^{k-1}\|^{2}_{H})+ C\int_{t_{k-1}}^{t_{k}} C_{g}|t_{k-1}-t'|^{2} \mathbb{E}(\|c(t')\|_{V}^{2}) +C \mathbb{E}(\| c(t')-c(t_{k-1})\|_{V}^{2})dt' \\
& \quad + C h^{2} \int_{t_{k-1}}^{t_{k}} \mathbb{E}( \| c(t')\|^{2}_{V}) dt'.
\end{align*}
Thanks to Assumption~\ref{lass-cholder} and using \eqref{regularityh} infer that
\begin{align}\label{term-II2}
II_{2} & \leq 
\frac{1}{4} \mathbb{E}\Big(\| ( e^{k}-e^{k-1}) \sqrt{|u^{k-1}_{hx}|}\|_{H}^{2} \Big) \\
& \quad +
C \, \triangle t \, \mathbb{E} \Big{(} \| e^{k-1} \sqrt{|u^{k-1}_{hx}|} \|^{2}_{H} \Big{)} + C ((\triangle t)^{2}+h^{2}) \int_{t_{k-1}}^{t_{k}} \mathbb{E}( \| c(t')\|^{2}_{V}) dt'
 + C(\triangle t)^{1+\nu_{r}}. 
 \notag
\end{align}

For term $III$ we write, using the definition of Ritz projection and $\varphi_{h}=e^{k}$,
\begin{align*}
III&=\int_{t_{k-1}}^{t_{k}} \langle \partial_{s} c(t'), \partial_{s} \varphi_{h} \rangle_{L^{2}(ds(t'))} dt' -\triangle t \langle (\mathcal{R}_{h}(c(t_{k})))_{x} \frac{1}{| u_{x}(t_{k})|}, \varphi_{hx} \rangle_{H} \\
&= \int_{t_{k-1}}^{t_{k}} \int_{S^{1}} (c(t'))_{x} \frac{e^{k}_{x}}{|u_{x}(t')|} dx dt'-
\int_{t_{k-1}}^{t_{k}} \int_{S^{1}} (c(t_{k}))_{x} \frac{e^{k}_{x}}{|u_{x}(t_{k})|} dx dt'\\
& =\int_{t_{k-1}}^{t_{k}} \int_{S^{1}} (c(t') )_{x} e^{k}_{x} ( \frac{1}{|u_{x}(t')|} - \frac{1}{|u_{x}(t_{k})|} ) dx dt'
+ \int_{t_{k-1}}^{t_{k}} \int_{S^{1}} ( c(t')-c(t_{k}))_{x} \frac{e^{k}_{x}}{|u_{x}(t_{k})|} dx dt.'
\end{align*}
Using \eqref{regularity}, the regularity of $u$, \eqref{regularityh}, a Young inequality, and Assumption~\ref{lass-cholder} we obtain
\begin{align}\label{term3}
\mathbb{E}(III) \leq \epsilon \, \triangle t \, \mathbb{E} \left (\left \| \frac{e_{x}^{k}}{\sqrt{|u^{k}_{hx}| } } \right \| ^{2}_{H} \right) + C_{\epsilon} \, \triangle t \,\int_{t_{k-1}}^{t_{k}} \mathbb{E} ( \|c(t')\|^{2}_{V} ) dt' +C_{\epsilon} (\triangle t)^{1+\nu_{r}}.
\end{align}

For term IV we write (as before, $\varphi_{h}=e^{k}$)
\begin{align*}
IV&=\langle c(t_{k}) |u_{x}(t_{k})| - c(t_{k-1}) |u_{x}(t_{k-1})|, \varphi_{h}\rangle_{H} -\langle \mathcal{R}_{h}(c(t_{k})) | u^{k}_{hx}| -\mathcal{R}_{h}(c(t_{k-1})) | u^{k-1}_{hx}|, \varphi_{h}\rangle_{H} \\
&=\langle \Big[ \big( c(t_{k}) - c(t_{k-1})\big) -\big(\mathcal{R}_{h}(c(t_{k})) - \mathcal{R}_{h}(c(t_{k-1})) \big) \Big]|u_{x}(t_{k})| , \varphi_{h}\rangle_{H} \\
& \quad+ \langle ( c(t_{k-1})-\mathcal{R}_{h}(c(t_{k-1}))) ( |u_{x}(t_{k})| - |u_{x}(t_{k-1})|), \varphi_{h}\rangle_{H} \\
& \quad+ \langle (\mathcal{R}_{h}(c(t_{k})) - \mathcal{R}_{h}(c(t_{k-1}))) (|u_{x}(t_{k})| - | u^{k}_{hx}|), \varphi_{h}\rangle_{H}
\\
&\quad+\langle \mathcal{R}_{h}(c(t_{k-1})) [( |u_{x}(t_{k})| - |u_{x}(t_{k-1})|)-( | u^{k}_{hx}| - | u^{k-1}_{hx}|)], \varphi_{h}\rangle_{H}\\
& = IV_{1}+IV_{2}+IV_{3}+IV_{4}.
\end{align*}
Note that $\| c(t) \|_V \leq \| c(t) - c(0) \|_V + \| c(0) \|_V$. As $c(0) \in V$, the regularity assumption \eqref{ass-cholder} yields that $\mathbb{E}(\| c(t) \|_V) \leq C$ at all times $t \in [0,T]$. Thanks to \eqref{RH4} we obtain that
\begin{align*}
IV_{1} \leq C_{g} h \| e^{k} \|_{H} \, \triangle t \, \| c(t_{k-1}) \|_{V}
+ C h \, \| c(t_{k}) - c(t_{k-1})\|_{V} \| e^{k}\|_{H}
\end{align*}
so that, using \eqref{ass-cholder} and recalling that $\nu_{r} \leq 1$,
\begin{align*}
\mathbb{E}(IV_{1}) \leq C h (\triangle t)^{\frac{\nu_{r}}{2}} \mathbb{E}(\| e^{k}\|_{H}^{2})^{\frac{1}{2}}.
\end{align*}
Using \eqref{Ritz2} and the regularity of $u$ we see that
\begin{align*}
IV_{2} \leq C h \, \| c(t_{k-1}) \|_{V} \, \triangle t \, \| e^{k}\|_{H} \quad
\end{align*}
so that
\begin{align*}
\mathbb{E}(IV_{2}) \leq C \, h \, \triangle t \, \mathbb{E}(\| e^{k}\|_{H}^{2})^{\frac{1}{2}}.
\end{align*}
Using \eqref{RH2} and the regularity of $u$ we can write
\begin{align*}
IV_{3} \leq Ch \Big(C_{g} \, \triangle t \, \| c(t_{k}) \|_{V} + C \| c(t_{k}) - c(t_{k-1})\|_{V} \Big) \| e^{k}\|_{H}
\end{align*}
so that, using \eqref{ass-cholder},
\begin{align*}
\mathbb{E}(IV_{3}) \leq C h (\triangle t)^{\frac{\nu_{r}}{2}} \mathbb{E}(\| e^{k}\|_{H}^{2})^{\frac{1}{2}}.
\end{align*}
For the last term of $IV$ we write
\begin{align*}
IV_{4}& =\langle \mathcal{R}_{h}(c(t_{k-1})) [( |u_{x}(t_{k})| - |u_{x}(t_{k-1})|)-( | u^{k}_{hx}|- | u^{k-1}_{hx}|)], e^{k}\rangle_{H}\\
&=\langle (c_{h}^{k-1}-e^{k-1})[( |u_{x}(t_{k})| - |u_{x}(t_{k-1})|)-( | u^{k}_{hx}|- | u^{k-1}_{hx}|)], e^{k}\rangle_{H}.
\end{align*}
Using the smoothness of $u$ as well as \eqref{regularity} and \eqref{regularityh}, and recalling that
a vector valued map $v$ satisfies $|v|_{t}=\frac{v}{|v|} \cdot v_{t}$, we have by Taylor expansions that
\begin{align*}
|u^{k}_{hx}(x)|&=|(I_{h}(u(t_{k},x))_{x}| = |\sum_{j}u(t_{k}, x_{j})\phi_{jx}(x)|\\
&=|\sum_{j}u(t_{k-1}, x_{j})\phi_{jx}(x)| +\frac{(\sum_{j}u(t_{k-1}, x_{j})\phi_{jx}(x))}{|\sum_{j}u(t_{k-1}, x_{j})\phi_{jx}(x)|}\cdot (\sum_{j}u_{t}(t_{k-1}, x_{j})\phi_{jx}(x)) \triangle t + O((\triangle t)^{2})\\
&=|u^{k-1}_{hx}(x)| + \frac{u^{k-1}_{hx}(x)}{|u^{k-1}_{hx}(x)|} \cdot (I_{h}(u_{t}(t_{k-1},x)))_{x}
\triangle t + O((\triangle t)^{2})
\end{align*}
and
\begin{align*}
|u_{x}(t_{k},x)| = |u_{x}(t_{k-1},x)| + \frac{u_{x}(t_{k-1},x)}{|u_{x}(t_{k-1},x)|}\cdot u_{xt}(t_{k-1},x) \triangle t + O((\triangle t)^{2}).
\end{align*}
Therefore, using
\eqref{regularity} and \eqref{regularityh} again,
\begin{align*}
|(& |u_{x}(t_{k},x)| - |u_{x}(t_{k-1},x)|)-( | u^{k}_{hx}(x)|- | u^{k-1}_{hx}(x)|)| \notag \\
& =\left|\frac{u^{k-1}_{hx}(x)}{|u^{k-1}_{hx}(x)|} \cdot (I_{h}(u_{t}(t_{k-1},x)))_{x} -\frac{u_{x}(t_{k-1},x)}{|u_{x}(t_{k-1},x)|}\cdot u_{xt}(t_{k-1},x) \right| \triangle t +O((\triangle t)^{2}) \notag \\
&\leq C h \triangle t + C(\triangle t)^{2} .
\end{align*}
Using Lemma~\ref{dis-apriori-g} for $c_{h}^{k-1}$ yields that
\begin{align*}
\mathbb{E}(IV_{4}) \leq C(h \triangle t + (\triangle t)^{2}) \Big{(} \mathbb{E}(\| e^{k}\|_{H}^{2})^{\frac{1}{2}} + \mathbb{E}(\| e^{k-1}\|_{H}^{2})^{\frac{1}{2}} \mathbb{E}(\| e^{k}\|_{H}^{2})^{\frac{1}{2}} \Big{)}.
\end{align*}
Putting all estimates together we obtain that 
\begin{align}\label{term-4}
\mathbb{E}(IV) &\leq C(h \triangle t + (\triangle t)^{2}) \Big{(} \mathbb{E}(\| e^{k}\|_{H}^{2})^{\frac{1}{2}} + \mathbb{E}(\| e^{k-1}\|_{H}^{2})^{\frac{1}{2}} \mathbb{E}(\| e^{k}\|_{H}^{2})^{\frac{1}{2}} \Big{)}
+ Ch (\triangle t)^{\frac{\nu_{r}}{2}} \mathbb{E}(\| e^{k}\|_{H}^{2})^{\frac{1}{2}} \notag\\
& \leq 
C(h \triangle t + (\triangle t)^{2}) \Big{(} 1 + \mathbb{E}(\| e^{k}\sqrt{|u^{k}_{hx}|} \|^{2}_{H}) + \mathbb{E}(\| e^{k-1}\sqrt{|u^{k-1}_{hx}|} \|^{2}_{H}) \Big{)}
\\
& \quad +Ch^{2} (\triangle t)^{\nu_{r}-1} + C \,\triangle t \, \mathbb{E} ( \| e^{k}\sqrt{|u^{k}_{hx}|} \|^{2}_{H}) \notag
\end{align}
where we have used \eqref{regularityh} in the last inequality.

Regarding the last term that we need to estimate we write
\begin{align*}
V&=- \triangle t \langle (\mathcal{R}_{h}(c(t_{k})))_{x} ( \frac{1}{| u^{k}_{hx}|} - \frac{1}{| u_{x}(t_{k})|}), \varphi_{hx} \rangle_{H} \\
&= \triangle t \int_{S^{1}} |e^{k}_{x}|^{2} ( \frac{1}{| u^{k}_{hx}|} - \frac{1}{| u_{x}(t_{k})|}) dx
- \triangle t \int_{S^{1}} (c^{k}_{h})_{x} e^{k}_{x}( \frac{1}{| u^{k}_{hx}|} - \frac{1}{| u_{x}(t_{k})|}) dx.
\end{align*}
Using \eqref{regularity}, \eqref{regularityh}, \eqref{(4.2)bis} together with the regularity of $u$ and a Young inequality, we obtain
\begin{align}\label{term5}
\mathbb{E}(V) \leq C \, \triangle t \, h \, \mathbb{E} \left (\left \| \frac{e_{x}^{k}}{\sqrt{|u^{k}_{hx}| } } \right \| ^{2}_{H} \right) + C \, \triangle t \, h \, \mathbb{E} \left (\left \| \frac{c^{k}_{hx}}{\sqrt{|u^{k}_{hx}| } } \right \| ^{2}_{H} \right) .
\end{align}

Finally, starting from \eqref{starting-eq}, integrating over $\Omega$, using \eqref{term6}, \eqref{term1}, \eqref{term-II1}, \eqref{term-II2}, \eqref{term3}, \eqref{term-4}, \eqref{term5},
choosing $\epsilon$, $h \leq h_{0}$, and $\triangle t \leq \triangle t_{0}$ sufficiently small we obtain that
\begin{align*} 
\mathbb{E} \Big( \frac{1}{2} &\| e^{k}\sqrt{|u^{k}_{hx}|} \|^{2}_{H} - \frac{1}{2} \| e^{k-1}\sqrt{|u^{k-1}_{hx}|}\|^{2}_{H} + \frac{1}{2} \|(e^{k} -e^{k-1})\sqrt{|u^{k-1}_{hx}|}\|^{2}_{H}
+\frac{1}{2}\triangle t \left \| \frac{e_{x}^{k}}{\sqrt{|u^{k}_{hx}|}} \right \|^{2}_{H} \Big) \notag \\
& \leq C \triangle t \,\mathbb{E} \Big{(} \| e^{k}\sqrt{|u^{k}_{hx}|} \|^{2}_{H} \Big{)} \\
& \quad + C ( (\triangle t)^{2} + h^{2}) \,\int_{t_{k-1}}^{t_{k}} \mathbb{E} ( \|c(t')\|^{2}_{V} ) dt' + C(\triangle t)^{1+\nu_{r}} + C \triangle t \, \mathbb{E} \Big{(} \| e^{k-1}\sqrt{|u^{k-1}_{hx}|}\|_{H}^{2} \Big{)}
\\ 
& \quad + 
\frac{2}{4} \mathbb{E}\Big(\|( e^{k}-e^{k-1}) \sqrt{|u^{k-1}_{hx}|}\|_{H}^{2} \Big) + 
C \, \triangle t \, \epsilon_{W} \\
& \quad 
+ C \triangle t \, \mathbb{E} \Big{(} \| e^{k-1}\sqrt{|u^{k-1}_{hx}|}\|_{H}^{2} \Big{)} + C ((\triangle t)^{2}+h^{2}) \int_{t_{k-1}}^{t_{k}} \mathbb{E}( \| c(t')\|^{2}_{V}) dt'
 + C(\triangle t)^{1+\nu_{r}} 
\\
& \quad + C \triangle t \,\int_{t_{k-1}}^{t_{k}} \mathbb{E} ( \|c(t')\|^{2}_{V} ) dt' + C (\triangle t)^{1+\nu_{r}} \\
& \quad
+ C (h \triangle t + (\triangle t)^{2}) \Big{(} 1 + \mathbb{E}(\| e^{k}\sqrt{|u^{k}_{hx}|} \|^{2}_{H}) + \mathbb{E}(\| e^{k-1}\sqrt{|u^{k-1}_{hx}|} \|^{2}_{H}) \Big{)}
\\
& \quad + C h^{2} (\triangle t)^{\nu_{r}-1} + C \, \triangle t \, \mathbb{E} \Big{(} \| e^{k} \sqrt{|u^{k}_{hx}|} \|^{2}_{H} \Big{)} \\
& \quad + C \, \triangle t \, h \, \mathbb{E} \left (\left \| \frac{c^{k}_{hx}}{\sqrt{|u^{k}_{hx}| } } \right \| ^{2}_{H} \right).
\end{align*}
Summing up for $k=1, \ldots, m \leq M$, using that $m (\triangle t) \leq M (\triangle t) =T$, Corollary~\ref{c-apriori-est}, Lemma~\ref{dis-apriori-g}, \eqref{regularityh},
we obtain that
\begin{align*}
\mathbb{E} \Big( \frac{1}{2 } &\| e^{m}\sqrt{|u^{m}_{hx}|} \|^{2}_{H} - \frac{1}{2 } \| e^{0}\sqrt{|u^{0}_{hx}|}\|^{2}_{H} \Big) \\
& \leq C (\triangle t + h \triangle t + (\triangle t)^{2} )\,\mathbb{E} ( \| e^{m}\sqrt{|u^{m}_{hx}|} \|^{2}_{H}) \\
& \quad + C \sum_{k=0}^{m-1}(\triangle t + h \triangle t + (\triangle t)^{2} )\,\mathbb{E} ( \| e^{k}\sqrt{|u^{k}_{hx}|} \|^{2}_{H}) \\
& \quad + C \epsilon_{W} + C (\triangle t + h)
+ C (\triangle t)^{\nu_{r}} + Ch^{2} (\triangle t)^{\nu_{r} -2}.
\end{align*}
By possibly decreasing $h_{0}$ and $\Delta t_{0}$ again
we can absorb the first term on the right hand-side and infer that
\begin{align*}
\mathbb{E} \Big{(} \|& e^{m}\sqrt{|u^{m}_{hx}|} \|^{2}_{H} \Big{)} \leq C \mathbb{E} \Big{(} \| e^{0}\sqrt{|u^{0}_{hx}|}\|^{2}_{H} \Big{)} +
C \sum_{k=0}^{m-1}(\triangle t + h \triangle t )\,\mathbb{E} \Big{(} \| e^{k}\sqrt{|u^{k}_{hx}|} \|^{2}_{H} \Big{)} \\
&+ C \epsilon_{W} + C (\triangle t + h)
+ C (\triangle t)^{\nu_{r}} + Ch^{2} (\triangle t)^{\nu_{r} -2}.
\end{align*}
Application of Gronwall (Lemma~\ref{discreteGronwall}) and using that $m \triangle t \leq T$ yields that
\begin{align*}
\mathbb{E} \Big{(} \| e^{m}\sqrt{|u^{m}_{hx}|} \|^{2}_{H} \Big{)} & \leq C \left( \mathbb{E} \Big{(} \| e^{0} \sqrt{|u^{0}_{hx}|}\|^{2}_{H} \Big{)} +
 \epsilon_{W} + (\triangle t)^{\nu_{r}} + h + h^{2} (\triangle t)^{\nu_{r} -2}\right).
\end{align*}
so that with $h=\triangle t$ we obtain that
\begin{align}\label{quasi-finale}
 \mathbb{E} (\| e^{m}\sqrt{|u^{m}_{hx}|} \|^{2}_{H}) &\leq C \left( \mathbb{E}(\| e^{0}\sqrt{|u^{0}_{hx}|}\|^{2}_{H})
 + \epsilon_{W} + (\triangle t)^{\nu_{r}} + h \right) .
\end{align}
The error $c^{m}_{h}- c(t_{m})$ is split into
\begin{align*}
c^{m}_{h}- c(t_{m})= (c^{m}_{h}- \mathcal{R}_{h}(c(t_{m}))) + (\mathcal{R}_{h}(c(t_{m}))- c(t_{m})) =
e^{m} + (\mathcal{R}_{h}(c(t_{m}))- c(t_{m})).
\end{align*}
Using \eqref{Ritz2} and $c(0) \in V$ we infer that
\begin{align*}
\| e^{0}\|_{H} \leq \|c^{0}_{h}- c(0)\|_{H} + \|\mathcal{R}_{h}(c(0))- c(0)\|_{H} \leq \|c^{0}_{h}- c(0)\|_{H}+Ch,
\end{align*}
as well as 
\begin{align*}
\|c^{m}_{k}- c(t_{m})\|_{H}^{2} \leq C\| e^{m }\|_{H}^{2}+ Ch^{2} (\| c(t_{m}) -c(0)\|_{V}^{2}+ \| c(0)\|_{V}^{2}).
\end{align*}
Together with \eqref{regularityh}, \eqref{quasi-finale}, and \eqref{ass-cholder} we obtain
\begin{align*}
\mathbb{E}(\|c^{m}_{h}- c(t_{m})\|_{H}^{2}) \leq C(\epsilon_{W} + (\triangle t)^{\nu_{r}} + h )
+ C \mathbb{E}( \|c^{0}_{h}- c(0)\|^{2}_{H})+Ch^{2}
\end{align*}
and the claim follows (using again that $h=\triangle t$). 
\end{proof}

\subsection{Remarks and refinements}
\label{sec:remarks}

We conclude the numerical analysis with some comments and generalisations.

\begin{rem}\label{rem3.11}
If we have higher regularity of the solution in space and time (for instance, when $B=0$) then we can recover standard error estimates
(cp. with \cite[Theorem~2.4]{DE12}). We exemplary discuss the treatment of term~$IV$
if the solution is differentiable with respect to time such that $c_{t} \in L^2(\Omega;L^2((0,T);W^{s,2}(S^{1})))$ for $s \in \{ 1,2 \}$.

First of all, noting that (without loss of generality we assume that $t > \tau$)
\begin{align}\label{buba}
\| z(t) - z(\tau) \|_{W^{s,2}(S^{1})} &=\|\int_{\tau}^{t} z_{t}(t') dt' \|_{W^{s,2}(S^{1})}\leq
\int_{\tau}^{t}\| z_{t}(t') \|_{W^{s,2}(S^{1})} dt' \notag \\
&\leq \sqrt{|t -\tau|} (\int_{\tau}^{t}\| z_{t}(t') \|_{W^{s,2}(S^{1})}^{2} dt')^{\frac{1}{2}}
\end{align}
we can replace \eqref{RH2} by 
\begin{align}\label{RH2-reg}
\| \mathcal{R}_{h} (z(t)) -\mathcal{R}_{h} (z(\tau)) \|_{H}& \leq C_{g} |t -\tau| \| z(\tau)\|_{V}+
 C \sqrt{|t -\tau|} (\int_{\tau}^{t}\| z_{t}(t') \|_{V}^{2} dt')^{\frac{1}{2}}
\end{align}
and, similarly, \eqref{RH4} by
\begin{align}\label{RH4-reg}
\| \mathcal{R}_{h} (z(t))& -\mathcal{R}_{h} (z(\tau)) - (z(t) -z(\tau)) \|_{H} \\
&\leq C_{g}|\tau-t|h^{s}\| z(\tau)\|_{W^{s,2}(S^{1})} + Ch^{s}
 \sqrt{|t -\tau|} (\int_{\tau}^{t}\| z_{t}(t') \|_{W^{s,2}(S^{1})}^{2} dt')^{\frac{1}{2}} .\notag
\end{align}
Thanks to the additional factor $\sqrt{|t -\tau|}$ no $\triangle t$ is ``lost'' in the summation during the Gronwall argument in the proof of Theorem~\ref{err-est-g}.
To see this, we estimate term $IV_{1}$ (defined before \eqref{term-4}) proceeding in a classical way: 
using the definition of the Ritz projection we write
\begin{align*}
IV_{1}&=\langle \Big[ \big( c(t_{k}) - c(t_{k-1})\big) -\big(\mathcal{R}_{h}(c(t_{k})) - \mathcal{R}_{h}(c(t_{k-1})) \big) \Big]|u_{x}(t_{k})| , e^{k}\rangle_{H} \\
&=\langle \Big[ \big( c(t_{k}) - c(t_{k-1})\big) -\big(\mathcal{R}_{h}(c(t_{k})) - \mathcal{R}_{h}(c(t_{k-1})) \big) \Big]|u_{x}(t_{k})| , e^{k}- (\frac{1}{|S^{1}|} \int_{S^{1}} e^{k} dx)\rangle_{H} \\
& \quad +\langle \Big[ \big( c(t_{k}) - c(t_{k-1})\big) -\big(\mathcal{R}_{h}(c(t_{k})) - \mathcal{R}_{h}(c(t_{k-1})) \big) \Big]|u_{x}(t_{k})| , (\frac{1}{|S^{1}|} \int_{S^{1}} e^{k} dx)\rangle_{H} \\
&=\langle \Big[ \big( c(t_{k}) - c(t_{k-1})\big) -\big(\mathcal{R}_{h}(c(t_{k})) - \mathcal{R}_{h}(c(t_{k-1})) \big) \Big]|u_{x}(t_{k})| , e^{k}- (\frac{1}{|S^{1}|} \int_{S^{1}} e^{k} dx)\rangle_{H} \\
& \quad -\langle \Big[ c(t_{k-1}) - \mathcal{R}_{h}(c(t_{k-1})) \Big](|u_{x}(t_{k})|- |u_{x}(t_{k-1})|) , (\frac{1}{|S^{1}|} \int_{S^{1}} e^{k} dx)\rangle_{H} \\
&\leq C\|\big( c(t_{k}) - c(t_{k-1})\big) -\big(\mathcal{R}_{h}(c(t_{k})) - \mathcal{R}_{h}(c(t_{k-1})) \big) \|_{H} \| e^{k}_{x}\|_{H}\\
&\quad+
\| \Big[ c(t_{k-1}) - \mathcal{R}_{h}(c(t_{k-1})) \Big](|u_{x}(t_{k})|- |u_{x}(t_{k-1})|) \|_{H} \| e^{k}\|_{H}
\end{align*}
where we have used Poincar\'{e} and \eqref{regularity} in the last inequality.
Applying \eqref{RH4-reg} and \eqref{Ritz2} yields that
\begin{align*}
IV_{1} & \leq \| e^{k}_{x} \|_{H} \Big(C \, \triangle t \, h^{s} \, \| c(t_{k-1}) \|_{W^{s,2}(S^{1})} + C \, h^{s} \sqrt{\triangle t} (\int_{t_{k-1}}^{t_{k}}\| c_{t}(t') \|_{W^{s,2}(S^{1})}^{2} dt')^{\frac{1}{2}} \Big)\\
& \quad + C \| e^{k} \|_{H} \, \triangle t \, h^{s} \, \| c(t_{k-1}) \|_{W^{s,2}(S^{1})} \\
& \leq \epsilon \triangle t\left (\left \| \frac{e_{x}^{k}}{\sqrt{|u^{k}_{hx}| } } \right \| ^{2}_{H} \right) + C_{\epsilon} \, h^{2s} \, \int_{t_{k-1}}^{t_{k}}\| c_{t}(t') \|_{W^{s,2}(S^{1})}^{2} dt' \\
& \quad + C \, \triangle t \, \| e^{k} \sqrt{|u^{k}_{hx}|} \|^{2}_{H}) + C_{\epsilon} \, h^{2s} \, \triangle t \, 
( \| c(t_{k-1}) - c(0)\|_{W^{s,2}(S^{1})}^{2} +\| c(0)\|_{W^{s,2}(S^{1})}^{2}).
\end{align*}
Observe that, in this case, coupling of spatial and time step size is not required.
\end{rem}

\begin{rem}
In Theorem~\ref{err-est-g} we had to couple the spatial with the time step sizes. This is due to the term $h^{2} (\triangle t)^{\nu_{r} -2}$ in the error estimate, which can be traced back to estimating term $IV$ (see \eqref{term-4} and before) and then from the manipulation
\begin{align*}
Ch (\triangle t)^{\frac{\nu_{r}}{2}} \mathbb{E}(\| e^{k}\|_{H}^{2})^{\frac{1}{2}} \leq
Ch^{2} (\triangle t)^{\nu_{r}-1} + C(\triangle t) \mathbb{E} ( \| e^{k}\sqrt{|u^{k}_{hx}|} \|^{2}_{H}) .
\end{align*}
In the summation step of Gronwall argument we ``loose'' another $\triangle t$ and arrive at the term $h^{2} (\triangle t)^{\nu_{r} -2}$.
If Assumption~\ref{lass-cholder} is replaced by the slightly stronger Assumption~\ref{lass-cholder-strong} below then we maintain the factor $h^{2} (\triangle t)^{\nu_{r} -1}$ also after summation in the Gronwall argument.
This does not affect the final estimation -- unless $\nu_{r}=1$: in that case, no coupling of spatial and time step sizes is required.

In fact, Assumption~\ref{lass-cholder-strong} allows us to incorporate the smooth case (see Remark~\ref{remNU1} and the higher regularity case discussed in Remark~\ref{rem3.11}) in one unique statement (see Theorem~\ref{err-est-g-2} below). Moreover it shows quite clearly that the ``limit case'' where a coupling of the time step with the grid size is needed takes place at $\nu_{r}=1$ and not at the value $\nu_{r}=2$.
\end{rem}
\begin{ass}\label{lass-cholder-strong}
There exists $\nu_{r} \in [0,1]$ and a positive map $\eta \in L^{1}(0,T)$ such that
 \begin{align}\label{ass-cholder-strong}
 \sup_{t, \tau \in [0,T], t < \tau} \frac{(\mathbb{E}(\| c(t) - c(\tau) \|_{V}^{2}))^{\frac{1}{2}}}{|t -\tau|^{\frac{\nu_{r}}{2} }} \leq (\int_{t}^{\tau} \eta(r) dr )^{\frac{1}{2}}.
 \end{align}
\end{ass}
\begin{teo}[Error estimates] \label{err-est-g-2}
Let $c$ be a solution according to Definition~\ref{def_solution} for some initial data $c(0)=c_{0} \in V$, and let Assumption~\ref{lass-cholder-strong} and Assumption~\ref{approximW} hold. Let $c_{h}^{k}$ be computed according to Algorithm~\ref{fullydiscrete-scheme-g}.
Further let $h \leq h_{0}$ and $\triangle t \leq \triangle t_{0}$ sufficiently small.
Let $\triangle t=h$ if $\nu_{r}<1$.
Then the following error estimate holds for any $k=1, \ldots, M$:
\begin{align*}
\mathbb{E}(\|c^{k}_{h}- c(t_{k})\|_{H}^{2})^{\frac{1}{2}}\leq C \sqrt{\epsilon_{W}} + C (\triangle t)^{\frac{\nu_{r}}{2}} + C \mathbb{E}( \|c^{0}_{h}- c_{0}\|_{H}^{2})^{\frac{1}{2}}+Ch.
\end{align*}
\end{teo}

\begin{rem}

In the deterministic case of Remark~\ref{rem3.11} we essentially recover standard estimates as given in \cite[Theorem~2.4]{DE12}. There, a much higher regularity in time is assumed. For our time-integrated formulation, an estimate of the form $\| c(t) - c(\tau)\|_{W^{2,s}(S^{1})} \leq C |t-\tau|$ for the solution in place of \eqref{buba} 
could be used to improve the time error estimate stated in Theorem~\ref{err-est-g-2}).

We also note that thanks to the time-integrated formulation the need for discussing a material derivative and its discrete counterpart naturally disappears. This leads to a significant simplification in the error analysis.
\end{rem}

Finally observe that, although we have been working with moving curves in the plane (in accordance with the applications we had in mind, which typically occur in a co-dimension one setting), the analysis presented applies seamlessly to the case of embedded curves in $\R^{n}$.

\section{Numerical Simulations and Convergence Assessments}

We study problems that are inspired by Example 10.43 in \cite{Powell} and consider a stochastic reaction-diffusion equations of the form
\[
 dc = \Big{(} - c \frac{|u_x|_t}{|u_x|} + D \frac{1}{|u_x|} \Big{(} \frac{c_x}{|u_x|} \Big{)}_x + r(c) \Big{)} dt + \frac{1}{|u_{x}|} B(c) dW.
\]
Note that this equation is of the form \eqref{eq:modeleq} with $w_{T}=0$. Upon testing with $\varphi \in V$ and integrating with respect to time we obtain \eqref{eq:def_solution} with an additional deterministic reaction term of the form $\int_0^t \langle r(c(t')) |u_x(t')|, \varphi \rangle_{H} dt'$ on the right-hand-side.

For two examples we assess the convergence and related our findings to the theoretical results in Theorem \ref{err-est-g}. Before that we first explain how we implement the noise and how we assemble and solve the algebraic problems in each time step.

\subsection{Noise approximation}

In our numerical experiments we choose $B$ in such a way that we effectively discretize a SPDE with noise defined by a $\hat{Q}$-Wiener process. 

Let us give the idea of our reasoning first, before diving into definitions and computations. We take $U=H$, and let $\hat{Q} \in L(U)=L(H)$ be a non-negative, symmetric operator of finite trace. There exists an orthonormal basis $g_{l}$, $l \in \N$ (for instance, see \cite[Proposition~2.1.5]{RoecknerBuch}) of $U$ such that
$$ \hat{Q}g_{l} = b_{l} g_{l}, \quad l \in \N $$
with $b_{l} \geq 0$ and $tr(\hat{Q}) = \sum_{l} b_{l} < \infty$. We can then define a $\hat{Q}$-Wiener process by
$$ \hat{W}(t) = \sum_{l=1}^{\infty} \sqrt{b_{l}} g_{l} \beta_{l}(t), \qquad t \in [0,T], $$
where the $\beta_{l}(t)$ are independent real valued Brownian motions (for instance, see \cite[Proposition 2.1.10]{RoecknerBuch} for more details).

Now define $B(c): U \to H$ by $B(c) g_{l} = \sigma(c) \sqrt{b_{l}} g_{l}$ for all $l \in \N$
with some Lipschitz continuous and bounded function $\sigma : \R \to [0,\infty)$.
If $\sigma$ is constant then $B = \sigma \hat{Q}^{\frac{1}{2}}$ is independent of the state variable $c$ and satisfies $\|B\|^{2}_{L_{2}(U,H)} = \sigma^2 tr(\hat{Q}) < \infty$ (see \cite[Proposition~2.3.4]{RoecknerBuch}). Note that then \eqref{BH3},\eqref{BH4}, \eqref{BH2} trivially are satisfied.
The stochastic forcing term \eqref{def_stochreact2} then reads (at least formally)
\begin{multline*}
\int_{0}^{t} \langle B(c(t')) dW(t'), \varphi \rangle_{H} = \sum_{l \in \N} \int_{0}^{t} \langle B(c(t')) g_{l}, \varphi \rangle_{H} d \beta_{l}(t') \\
=\sum_{l \in \N} \int_{0}^{t} \langle \sigma \sqrt{b_{l}} g_{l}, \varphi \rangle_{H} d \beta_{l}(t') = \langle \int_{0}^{t} \sigma d\hat{W}(t'), \varphi \rangle_{H}.
\end{multline*}

In our experiments we choose $\hat{Q}$ to be defined through the following orthonormal basis for $H$:
\begin{align}\label{gBasis}
 g_1(x) = \frac{1}{\sqrt{2\pi}}, \quad g_{2n}(x) = \frac{1}{\sqrt{\pi}} \sin(nx), \quad g_{2n+1}(x) = \frac{1}{\sqrt{\pi}} \cos(nx), \qquad n \in \N.
\end{align}
Note that the maps $g_{l}$, $l \in \N$, are the $L^2$-normalized eigenfunctions of the operator $Av = - v_{xx}$ with periodic boundary conditions.

With regards to the coefficients we choose $b_{l} = n^{-2\bar{r}-1}$, $l \in \{ 2n, 2n+1 \}$, in our computations, where $\bar{r}>0$ is a decay rate. We are not able to theoretically verify the regularity assumption~\eqref{ass-cholder}. However, we note that \cite[Theorem 2.31]{Kruse14} establishes temporal regularity for mild solutions to semilinear stochastic evolution equations, for which \eqref{ass-cholder} is satisfied, and which therefore motivates our choices of the $b_l$ and $\bar{r}$ in the examples below.

\subsection{Matrix-vector formulation}

For simplicity we choose equidistant nodes $x_{j}$, so that $h_{j}=h$ for all $j$.
Recall that, in each time step, the system \eqref{eq:linsys} has to be solved, which is obtained by choosing $\varphi_{h} = \phi_i$ in \eqref{eq-3.6-g}. Let $q^{k}_{j} = |u(t_{k},x_{j}) - u(t_{k},x_{j-1})|$ for all $j$ (spatial indices modulo $N$) and $k$. Short calculations show that
\[
 M^{k}_{i,j} =
 \begin{cases}
 (q_{j}^{k} + q_{j+1}^{k}) / 3, & \\
 q_{j}^{k}/6, & \\
 q_{j+1}^{k}/6, & \\
 0, &
 \end{cases}
 \qquad
 S^{k}_{i,j} =
 \begin{cases}
 1/q_{j}^{k} + 1/q_{j+1}^{k} & \qquad \quad i=j, \\
 -1/q_{j}^{k} & \qquad \quad i=j-1, \\
 -1/q_{j+1}^{k} & \qquad \quad i=j+1, \\
 0 & \qquad \quad \mbox{otherwise. }
 \end{cases}
\]
If $\sigma(c)$ is not constant then we make the approximation
\begin{multline*}
 \langle \Bapprox(c^{k-1}_{h}) \triangle W_{k}, \phi_{i} \rangle_{H}
 = \sum_{l=1}^{L} \langle \sigma(c^{k-1}_{h}) \sqrt{b_{l}} g_{l}, \phi_{i} \rangle_{H} \triangle \beta_{l,k} \\
 \approx \sum_{l=1}^{L} \langle \sqrt{b_{l}} g_{l}, I_{h} (\sigma(c^{k-1}_{h}) \phi_{i}) \rangle_{H} \triangle \beta_{l,k}
 = \sum_{l=1}^{L} \langle \sqrt{b_{l}} g_{l}, \sigma(c_{i}^{k-1}) \phi_{i} \rangle_{H} \triangle \beta_{l,k}.
\end{multline*}
If $l = 2n$ then
\[
 \sqrt{\pi} \langle g_{l}, \phi_{i} \rangle_{H} = \langle \sin(n \cdot), \phi_{i} \rangle_{H}
 = \frac{1}{n^2 h} \big{(} - \sin(n x_{i-1}) + 2 \sin(n x_{i}) - \sin(n x_{i+1}) \big{)},
\]
and, similarly, if $l = 2n+1$ (then $\sin$ is replaced by $\cos$)
for any $n \in \N$.

Assuming that $w_T = 0$ we therefore obtain for the right-hand-side of the system \eqref{eq:linsys} that
\begin{align*}
f_{i}
&= \langle c^{k-1}_{h} | u^{k-1}_{hx}|, \phi_{i} \rangle_{H}
+ \langle \Bapprox(c^{k-1}_{h}) \triangle W_{k}, \phi_{i} \rangle_{H} \\
&= \sum_{j=1}^N M^{k-1}_{i,j} \cdot c^{k-1}_{j}
+ \sum_{l=1}^L \langle g_{l}, \phi_{i} \rangle_{H} \sqrt{b_{l}} \sigma(c_{i}^{k-1}) (\beta^{k}_{l} - \beta^{k-1}_{l}).
\end{align*}

\subsection{Convergence as $\triangle t \to 0$}
\label{sec:LPSevolving}

\begin{figure}
\begin{center}
 \includegraphics[width=15cm]
 {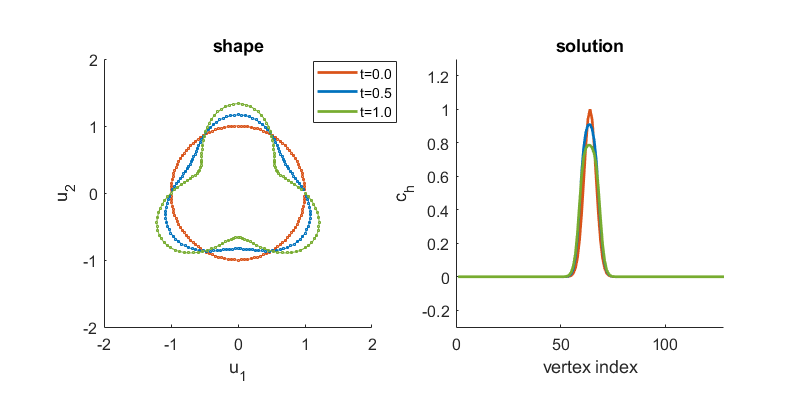}
\end{center}
\caption{Illustration of the solution for the problem in Section \ref{sec:LPSevolving}. The shape of the evolving curve $u$ is displayed at several times on the left using $N = 128$ vertices. On the right, a numerical sample path has been computed with the parameters $\triangle t = 10^{-4}$, $L = 259$, $\bar{\sigma} = 0.5$, $\bar{r} = 1.0$ and is displayed at the same times as the shape. Note that the solution is plotted over the vertex indices $i=1,\dots N$ rather than the corresponding vertex positions $x_i = 2 \pi \, i/N$ in the spatial domain.
}
\label{fig:LPSevolving}
\end{figure}

In the case of sufficient noise we expect to see convergence rates $\frac{\nu_r}{2} < \frac{1}{2}$ as predicted in Theorem \ref{err-est-g}.
This is due to the time discretization, noting that $h$ features with the power one on the right-hand-side in the estimate. To assess this expectation we therefore fix the spatial mesh and compute convergences rates as the time step size $\triangle t \to 0$.

The specific problem that we consider for the purpose is inspired by Example 10.43 in \cite{Powell} but on a closed evolving curve. We set $D = 0.001$ and $r(c) = c (1-c) (c+0.5)$. The parametrisation of the curve is given by
\[ 
u(t,x) = (1-t/3) \sin(3x) (\cos(x),\sin(x)), \quad (t,x) \in [0,T] \times [0,2\pi).
\] 
The parameters in the noise process are motivated by Example 10.10 in \cite{Powell}.
Recalling \eqref{gBasis} and that $B(c)g_{l} = \sigma(c) \sqrt{b_{l}} g_{l}$ we set $\sigma(c) = \max \{ \bar{\sigma} c(1-c), 0 \}$,
$b_{1} = 1$, and $b_{l} = n^{-2\bar{r}-1}$, $l \in \{ 2n, 2n+1 \}$, $n \in \N$. For the noise strength and noise decay parameters we choose $\bar{\sigma} = 0.5$ and $\bar{r} = 1.0$, respectively. The initial data are deterministic and given by
\[
 c_{0}(x) = \exp \big{(} - \tfrac{1000}{4 \pi^2} (x-\pi)^2 \big{)}, \quad x \in [0,2\pi).
\]
For the computations these are interpolated on the spatial mesh, i.e., $c_{h}^0 = I_h (c_{0})$. Figure \ref{fig:LPSevolving} gives an impression of the evolving geometry and the solution.

We compute $S = 100$ (the results below are robust with respect to this choice) samples paths of a reference solution $c_{ref}$ on a fixed mesh with $N = 128$ vertices, with $\triangle t_{ref} = 10^{-5}$. 

With regards to the noise truncation parameter $L$ we note that, usually, (for instance, see \cite{Powell,BDE}) the dimension of the finite element space is used ($L = N$). This is motivated by the fact that frequencies up to that dimension can be resolved by the finite element mesh. However, we here have periodic boundary conditions so that both $\sin$ and $\cos$ functions feature. As a consequence, the frequency of $g_l$ is half the index ($l/2$ or $(l-1)/2$). We therefore choose twice as many noise terms, more precisely, $L = 2N+1 = 259$. We have performed computations with other values and generally observed that this is a good cut-off.

The solutions at the final time $T = 1$ after $M_{ref} = T / \delta t_{ref} = 10^5$ steps are denoted by $c_{ref}^{M_{ref}}(\omega_s)$ where $\omega_s \in \Omega$ stands for the $s$-th sample path.
Next, we compute the same sample paths for the time step sizes $\triangle t = p \triangle t_{ref}$ with $p \in \{ 500, 200, 100, 50, 20, 10, 5 \}$ and, for each sample path at the final time $c_{h}^{M}(\omega_s)$, $M = T/\triangle t$, the $L^2$ distance to the corresponding sample path of the reference solution. We use the average as a measure for the error due to the time step size:
\begin{equation} \label{eq:def_num_err}
 E_S(\triangle t) = \frac{1}{S} \sum_{s=1}^S (\| c^{M}_{h}(\omega_s) - c^{M_{ref}}_{ref} (\omega_s) \|_{H}^{2})^{\frac{1}{2}}.
\end{equation}

\begin{table}
\begin{center}
\begin{tabular}{|l|l|l|} \hline
$\triangle t$ & $E_S$ & $\eoc$ \\ \hline \hline
     0.005   &  4.4034e-04  &         -- \\ \hline
     0.002   &  2.6045e-04  &    0.57309 \\ \hline
     0.001   &  1.8081e-04  &    0.52659 \\ \hline
     0.0005  &  1.2459e-04  &    0.53728 \\ \hline
     0.0002  &  7.9164e-05  &    0.49492 \\ \hline
     0.0001  &  5.1203e-05  &    0.62861 \\ \hline
     5e-05   &  3.7661e-05  &    0.44315 \\ \hline
\end{tabular}
\end{center}
\caption{Errors and experimental orders of convergence for the example describes in Section \ref{sec:LPSevolving}.
The error $E_S$ is given by \eqref{eq:def_num_err}.
}
\label{tab:conv_LPSevolving}
\end{table}

Table \ref{tab:conv_LPSevolving} displays the errors and the corresponding experimental orders of convergence. These indeed are around $0.5$ as in \cite[Example 10.43]{Powell}. In contrast, the eocs are around $1.0$ in the deterministic case ($\bar{\sigma} = 0$).

\subsection{Convergence as $h \sim \triangle t \to 0$}
\label{sec:wavefront_shrinking}

\begin{figure}
\begin{center}
 \includegraphics[width=15cm]
 {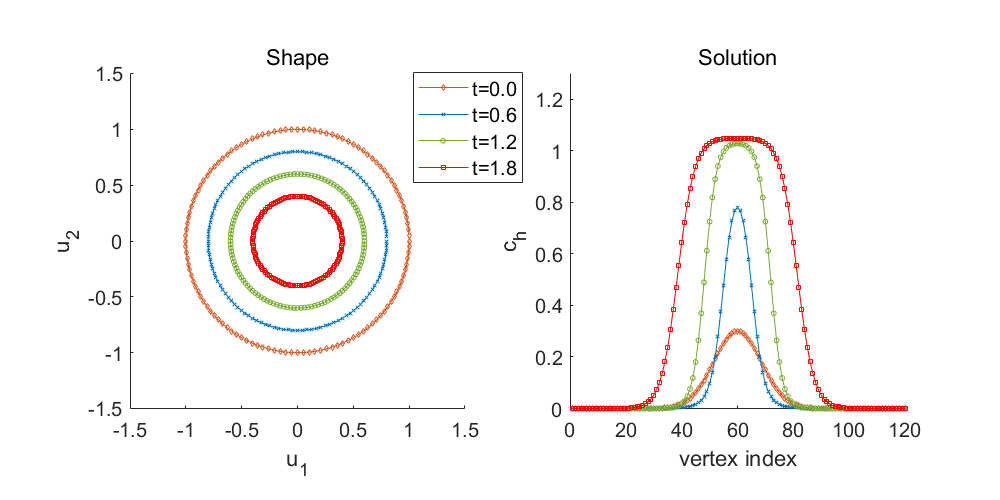}
\end{center}
\caption{Illustration for the problem in Section \ref{sec:wavefront_shrinking}. The shape of the evolving curve $u$ is a shrinking circle
and displayed at several times on the left.
On the right, the numerical solution ($N=120$, $\triangle t = 10^{-4}$) of the deterministic ($\bar{\sigma}=0$) problem at the same times is displayed. Note that it is plotted over the vertex indices $i=1,\dots N$ rather than the corresponding vertex positions $x_i = 2 \pi \, i/N$ in the spatial domain.
The solution indicates that the initial signal $c_{0}$ is amplified and then leads to outwards moving fronts between domains where the solution is close to zero or one, respectively.
}
\label{fig:Deterministic}
\end{figure}

\begin{figure}
\begin{center}
 \includegraphics[width=7cm]
 {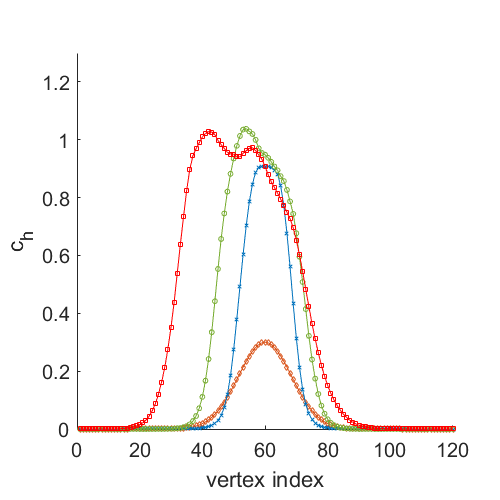}
 \includegraphics[width=7cm]
 {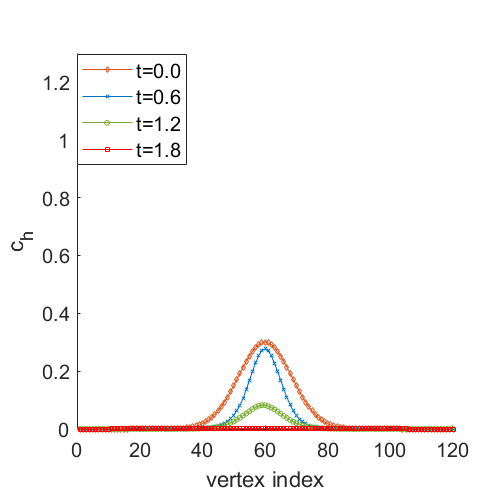}
\end{center}
\caption{Two numerically computed sample paths for the stochastic PDE in Section \ref{sec:wavefront_shrinking} plotted of the vertex indices $i=1, \dots, N$. The parameters are $\bar{r} = 0.75$, $\bar{\sigma} = 0.5$ $N = 120$, $\triangle t = 10^{-4}$, $L = 301$. We observe that the initial signal can get amplified with a domain being formed where the value of $c_h$ is around one (left). Fronts between the two domains also still are noticeable. However, it can also happen that $c_h$ vanishes (right).}
\label{fig:Stochastic}
\end{figure}

We now investigate the convergence behaviour if both the spatial and the time step size decay at the same rate, as it is required for the estimate in Theorem \ref{err-est-g}. We consider a different problem by setting $D = 0.05$ and $r(c) = 20 c(1-c)(c-0.25)$.
For the shape we consider a self-similarly shrinking circle,
\[ 
u(t,x) = (1-t/3) (\cos(x),\sin(x)), \quad (t,x) \in [0,T] \times [0,2\pi),
\] 
so that $|u_x(t,x)| = 1-t/3$.
Regarding the parameters in the noise process we set $b_{1} = 0$, and $b_{l} = n^{-2\bar{r}-1}$, $l \in \{ 2n, 2n+1 \}$. Furthermore $\sigma(c) = \min \{ \max \{ \bar{\sigma} c, 0 \}, 100 \}$. The values of $\bar{\sigma} \geq 0$ and $\bar{r} > 0$ vary.
We chose the (deterministic) initial data
\[
 c_{0}(x) = 0.3 \exp \big{(} - \tfrac{100}{4 \pi^2} (x-\pi)^2 \big{)}, \quad x \in [0,2\pi)
\]
and interpolate them to start the computations.

Solutions to the deterministic equation often form large patches where $c \approx 0$ or $c \approx 1$ that are separated by layers moving such that those where $c \approx 1$ increase. For the specific initial data, Figure \ref{fig:Deterministic} gives an impression of the solution. If we add the multiplicative noise term then often a domain where $c \approx 1$ and layers still can be observed, but it can also happen that $c$ vanishes in the long run. Figure \ref{fig:Stochastic} displays two sample paths to give an idea of possible outcomes. These computations were done with $N = 120$ mesh points $x_i = 2 \pi \, i/N$, $i=1,\dots,N$,
and on the time interval $[0,T] = [0,1.8]$ with time step size $\triangle t = 0.001$ and, in the noisy cases, with $\bar{\sigma} = 0.5$, $\bar{r} = 0.75$, and $L = 301$.

We proceed as previously and use reference solutions to assess the errors. These are computed with two different spatial step sizes $h_{ref} = 2 \pi / N_{ref}$, namely for $N_{ref} \in \{ 1200, 9600 \}$. We choose $L = 2N_{ref}+1$ for the truncation of the noise in all computations (this choice was motivated in the previous section \ref{sec:LPSevolving}). Setting the final time to $T = 0.6$ we use $M_{ref} = 600$ time steps so that $\triangle t_{ref} = 0.001$. We write $c_{ref}^{M_{ref}}(\omega_s)$ for the $s$-th sample path, $s = 1, \dots, S$ where $S = 100$. The same samples paths then are computed again for the step sizes $\triangle t = p \triangle t_{ref}$ and $h = p h_{ref}$ for $p \in \{ 3, 4, 5, 6, 8, 10 \}$. This means that, if $N_{ref} = 1200$ then $h \approx 5.236 \triangle t$, and if $N_{ref} = 9600$ then $h \approx 0.655 \triangle t$. The errors are approximated as in \eqref{eq:def_num_err} (average of the $L^2$ distance to the reference solution at the final time) where we interpolate the numerical solutions to the reference mesh and compute the $L^2$ integral exactly (modulo rounding errors).

\begin{table}
\begin{center}
\begin{tabular}{|l|l|l|l|} \hline
$h$ & $\triangle t$ & $E_S$ & $\eoc$ \\ \hline \hline
     0.05236  &       0.01  &     0.012407  &         --  \\ \hline
     0.041888 &       0.008 &     0.0092946 &      1.2944 \\ \hline
     0.031416 &       0.006 &     0.0075923 &      0.7032 \\ \hline
     0.02618  &       0.005 &     0.0058762 &      1.4053 \\ \hline
     0.020944 &       0.004 &     0.0045545 &      1.1418 \\ \hline
     0.015708 &       0.003 &     0.0034464 &      0.9691 \\ \hline
\end{tabular} \hfill
\begin{tabular}{|l|l|l|l|} \hline
$h$ & $\triangle t$ & $E_S$ & $\eoc$ \\ \hline \hline
     0.006545 &       0.01  &     0.011159  &         --  \\ \hline
     0.005236 &       0.008 &     0.010083  &      0.4544 \\ \hline
     0.003927 &       0.006 &     0.0077323 &      0.9228 \\ \hline
     0.003273 &       0.005 &     0.0061825 &      1.2268 \\ \hline
     0.002618 &       0.004 &     0.0055465 &      0.4865 \\ \hline
     0.001964 &       0.003 &     0.0040649 &      1.0803 \\ \hline
\end{tabular}
\end{center}
[\caption{Errors and experimental orders of convergence for the example describes in Section \ref{sec:wavefront_shrinking}. Left table: results for $N_{ref} = 1200$. Right table: results for $N_{ref} = 9600$. The error $E_S$ is given by \eqref{eq:def_num_err}.
}
\label{tab:conv_wavefront_shrinking}
\end{table}

Table \ref{tab:conv_wavefront_shrinking} displays the errors and eocs. The latter generally seem a bit higher if $N_{ref} = 1200$. The errors support this observation. If $N_{ref} = 1200$ then halving the step size more than halves the error; for instance, it is $0.0092946$ for $\triangle t = 0.008$ and becomes $0.0045545$ for $\triangle t = 0.004$. In turn, if $N_{ref} = 9600$ then the error does not quite halve when halvening the step size; for instance, from $0.010083$ for $\triangle t = 0.008$ it goes down to $0.0055465$ for $\triangle t = 0.004$ only.

Recalling again the convergence result in Theorem \ref{err-est-g} the convergence is at most linear in $h$ and $\frac{\nu_r}{2} < 1$ in $\triangle t$. We interpret the above findings as follows. If $h$ is relatively large with respect to $\triangle t$ (such as in the case $N_{ref} = 1200$ then the spatial discretization error is dominating so that linear convergence is observed as $h \sim \triangle t \to 0$ at the same rate. In turn, if $h$ is relatively small in comparison to $\triangle t$ then the time discretization error is dominating so that slower convergence is observed. In the case $N_{ref} = 9600$ we are in that convergence regime.

Finally, we also want to get an idea how likely it is that the initial signal vanishes in the long run as on the right in Figure \ref{fig:Stochastic}. We perform $S = 1000$ simulations for the same data ($\bar{r} = 0.75$, $\bar{\sigma} = 0.5$, $N = 120$, $\triangle t = 10^{-4}$, $L = 301$). If the spatial $L^2$ norm a the final time $T = 1.8$ of the sample is $\leq 0.1$ we deem the signal to have vanished. In our computations this happened in 22.2\% of the cases.

\bibliography{refs}
\bibliographystyle{acm}

\end{document}